\documentclass{article}
\usepackage{graphicx}
\usepackage{amsmath}
\usepackage{amssymb}
\usepackage{tikz-cd}
\usepackage{amsthm}
\usepackage{mathrsfs}
\usepackage{fullpage}
\usepackage{float}

\newcommand{\R}{\mathbb{R}}

\newcommand{\E}{\mathbb{E}}

\newtheorem{thm}{Theorem}[subsection]
\renewcommand{\thethm}{%
	\ifnum\value{subsection}>0
	\thesubsection
	\else
	\thesection
	\fi
	.\arabic{thm}%
}
\newtheorem{lemma}[thm]{Lemma}

\newtheorem{cor}[thm]{Corollary}

\newtheorem{prop}[thm]{Proposition}

\newtheorem{conj}[thm]{Conjecture}

\theoremstyle{definition}
\newtheorem{Def}[thm]{Definition}

\newtheorem{rk}[thm]{Remark}

\title{Left-Invariant Riemannian Distances on Higher-Rank Sol-Type Groups}
\author{Daniel N. Levitin}
\date{\today}

\begin{document}
	
	\maketitle
	
	\begin{abstract}
		In this paper, we generalize the results of Le Donne, Pallier, and Xie (\textit{Groups, Geom. Dyn.}, \textbf{19} (2025), 227--263) to describe the split left-invariant Riemannian distances on higher-rank Sol-type groups $G=\mathbf{N}\rtimes \R^k$. We show that the rough isometry type of such a distance is determined by a specific restriction of the metric to $\R^k$, and therefore the space of rough similarity types of distances is parameterized by the symmetric space $SL_k(\R)/SO_k(\R)$. In order to prove this result, we describe a family of uniformly roughly geodesic paths, which arise by way of the new technique of \textit{Euclidean curve surgery}.
	\end{abstract}
	
	\section{Introduction}
	
	Describing the geometry of solvable Lie groups and their lattices, the discrete virtually polycyclic groups, is one of the major areas of work in geometric group theory. Solvable Lie groups admit an algebraic classification due to a theorem of Auslander \cite{Auslander1}, which states that every simply-connected solvable Lie group $G$ fits into a short exact sequence 
	\[ 1\to N_1\to G\to N_2\to 1,\]
	where $N_1$ and $N_2$ are simply-connected nilpotent Lie groups. See \cite{EskinFisher} for a detailed description of what is known and supposed about these groups. Much of the work of describing the geometry of solvable Lie groups has focused on cases where the defining short exact sequence splits, and where the group $N_2$ is abelian. In this paper, we focus on \textit{higher-rank Sol-type groups} $\prod_{i=1}^n \mathbf{N}_i \rtimes \R^k$, where the groups $\mathbf{N}_i$ are simply-connected nilpotent, $k$ is at least $2$, and the group $\R^k$ acts on each space $\mathbf{N}_i$ by contraction. See Definitions \ref{Def:Derivation} and \ref{Def:SolTypeGp} for a complete description.
	
	One set of coarse geometric questions concerns the structure of various groups of self-mappings. A map $f:(X, d_X)\to (Y, d_Y)$ between metric spaces is a \textit{quasi-isometry} if there are positive constants $A$ and $B$, and non-negative $C$ so that, for all $x_1$ and $x_2$ in $X$,
	\[ A d_X(x_1, x_2)-C\le d_Y(f(x_1), f(x_2))\le B d_X(x_1, x_2)+C,\]
	
	\noindent and if, in addition, each $y$ in $Y$ is within $C$ of a point in $f(X)$. If $A=B$, then such a map is a \textit{rough similarity}. If $A=B=1$, then such a map is a \textit{rough isometry}. Finally, if $C=0$, then such a map is a \textit{bi-Lipschitz equivalence}.
	
	If $(X, d_X)=(Y, d_Y)$, then the collection of maps of these types, modulo the equivalence $f\sim g$ if $\sup_{x \in X} d_Y(f(x), g(x))<\infty$, are groups under composition, denoted respectively $QI(X, d_X)$, $RS(X, d_X)$, $RI(X, d_X)$, and $BiLip(X, d_X)$. One asks, for a given solvable Lie group $X$ and left-invariant Riemannian distance $d_X$, if any structure is preserved by these groups, and if so, whether it is sufficient to determine the group. For instance, the work of Eskin--Fisher--Whyte, and later of Peng, provides tools to describe specific metrics on nondegenerate Sol-type groups \cite{EFW1, EFW2, Peng1, Peng2}, and to compute the groups of quasi-isometries.
	
	One may also ask finer geometric questions. For instance, we often wish to know what the efficient paths with respect to a specific left-invariant Riemannian metric may look like, and how they may differ from efficient paths with respect to some other metric. In a similar vein, while any pair of left-invariant Riemannian distances will always be bi-Lipschitz, they need not be roughly similar or roughly isometric. One fruitful direction of research has been to investigate, of a given space and class of metrics, how much information is required to compute the rough similarity or rough isometry class of a metric, and what space these metrics parameterize (see, e.g., \cite{BaderFurman, Coornaert, Furman, OregonReyes}).
	
	For \textit{Sol-type groups}, which are defined like their higher-rank counterparts, but with $k=1$, Le Donne--Pallier--Xie have shown that all left-invariant Riemannian distances are roughly similar, and the rough isometry class of a distance is determined by its induced metric on the $\R$-factor \cite{LDPX}. In this paper, we achieve an equivalent result for the subset of \textit{perpendicularly split} left-invariant Riemannian metrics. These are the metrics $g$ for which the $\R^k$ factor can be taken to be $g$ perpendicular to the $\prod_{i=1}^n \mathbf{N}_i$.
	
	\begin{thm} \label{IntroRIThm}
		
		Let $G=\prod_{i=1}^n \mathbf{N}_i\rtimes \R^k$ be a higher-rank Sol-type group. Let $g_1$ and $g_2$ be perpendicularly split left-invariant Riemannian metrics, and let their induced metrics on the $\R^k$ factors be equal (resp. similar). Then the related left-invariant Riemannian distances $d_1$ and $d_2$ are roughly isometric (resp. roughly similar).
		
	\end{thm}
	
	In rank-1, the perpendicular splitting assumption, which is used throughout \cite{LDPX}, is immediate from the group theory. However, in higher rank, only a positive-codimension set of left-invariant Riemannian metrics are perpendicularly split.
	
	Theorem \ref{IntroRIThm} arises as a special case of a stronger result, which is our main theorem.
	
	\begin{thm} {[Main Theorem]} \label{EigenvaluesAndDistancesIntro}
		
		Let $G=\prod_{i=1}^n \mathbf{N}_i\rtimes \R^k$ be a higher-rank Sol-type group. Let $g_1$ and $g_2$ be metrics for which $G$ splits perpendicularly, and let their induced metrics on the $\R^k$ factors differ with top and bottom eigenvalues $\lambda_1$ and $\lambda_k$ respectively. Let $g_1$ and $g_2$ give rise to distances $d_1$ and $d_2$. Then there is a constant $C>0$ so that
		\[ \lambda_kd_1 -C\le d_2\le \lambda_1 d_1+C.\]
	\end{thm}
	
	The collection of metrics up to rough similarity on $G$ is therefore parameterized by the space of metrics up to rough similarity on $\R^k$. We will describe in Definition \ref{Def:SpaceOfMetrics} how to geometrize this space. With respect to this metric, we obtain the following theorem.
	
	\begin{thm} \label{SpaceOfMetricStructures}
		
		The space of left-invariant Riemannian distances that split $G$ perpendicularly, modulo rough similarity, $\mathscr{D}_{Riem, \perp}$ is isometric to $SL_k(\R)/SO_k(\R)$.
		
	\end{thm}
	
	This generalizes Le Donne-Pallier-Xie's theorem that all left-invariant Riemannian distances on rank-1 Sol-type groups are all roughly similar, as $SL_1(\R)/SO_1(\R)$ consists of a single point.
	
	Finally, we obtain a corollary concerning the metric properties of self quasi-isometries of nondegenerate higher-rank Sol-type groups for some groups. 
	
	\begin{thm}\label{IntrinsicRIsThmIntro}
		Let $G$ be a nondegenerate unimodular abelian-by-abelian higher-rank Sol-type group. Let $d$ be any left-invariant distance on $G$. There is a finite-index subgroup of $QI(G,d)$ consisting of maps that are rough isometries for every perpendicularly split left-invariant Riemannian distance.
	\end{thm}
	
	See Theorem \ref{IntrinsicRIs} for the precise statement of this theorem. It seems likely that similar results should hold for any nondegenerate higher-rank Sol-type group. This would follow from Conjecture \ref{GeneralizedPengQIs}, which asserts that the results of Peng in \cite{Peng1, Peng2} go through  in this level of generality. 
	
	\subsection*{Outline of the paper}

	In Section \ref{sec:Preliminaries} we state the definitions we will use throughout the paper. Section \ref{sec:LieAlgebraicProperties} is devoted to generalities about the Lie algebraic properties of all splittings of the quotient $G\to \R^k$.
	
	We then set about proving the main theorem by an argument analogous to the proof in rank-1 due to \cite{LDPX}. In Section \ref{sec:GeometryOfSplittings} we describe geometric features of left-invariant Riemannian distances that admit perpendicular splittings. In Section \ref{sec:BoxGeodesicLengths}, we introduce to each path a set of \textit{half spaces} and a distance approximation denoted $\rho$, which records the shortest length of a path that visits each half-space (see Definition \ref{HSV} for details). Half-space visiting paths are particularly easy to understand and compare to one another. In Proposition \ref{EigenvalueDistanceUpperBoundProp} we compute the length of a $g_1$-half-space visiting curve with respect to a different metric, $g_2$. This provides the upper bound in Theorem \ref{EigenvaluesAndDistancesIntro}, with $d_1$ replaced with its $\rho$-approximation, and a similar lower bound. 
	
	Therefore, to prove Theorem \ref{EigenvaluesAndDistancesIntro} it suffices to find a bound, depending on $g$, on the difference between the corresponding distance and $\rho$-function, i.e., to show that the shortest half-space visiting path between two points is never more than a constant length longer than the shortest path of any kind. In \cite{LDPX}, this is proven by showing that geodesics must come within a bounded distance meeting each half-space. However, this proof does not go through in higher rank.
	
	Instead, we will take the $\R^k$-coordinates of a geodesic, and modify it in a controlled way to visit each half space. In Section \ref{EuclideanSurgery}, we introduce the technique of \textit{Euclidean curve surgery}. We describe how to change a curve in Euclidean space simultaneously in several different ways without lengthening it very much. Then in Section \ref{sec:DiagonalizableCaseCoarseGeodesics} we apply curve surgery to the $\R^k$-coordinate of our geodesic. The result is Proposition \ref{DiagonalizableHSVProp} which show that, without increasing the length of the geodesic much, we can make it visit each half-space.
	
	Finally, in Section \ref{sec:IntrinsicRIs}, we prove Theorem \ref{IntrinsicRIsThmIntro} as a corollary of the description of efficient paths obtained in Proposition \ref{DiagonalizableHSVProp}.
	
	\subsection*{Acknowledgments}
	
	I would like to thank my advisor, Tullia Dymarz, for many helpful suggestions for directions of investigation. I would also like to thank Enrico Le Donne for discussing this project with me, Eduardo Reyes for a conversation that led to Theorem \ref{SpaceOfMetricStructures}, and especially Gabriel Pallier for several suggestions to improve the paper and a conversation that led to Theorem \ref{IntrinsicRIsThmIntro}.
	
	I was supported by NSF grant 2230900 during this project. The beginning of this project took place at the Fields Institute, and I am grateful for its support.

	\section{Preliminaries} \label{sec:Preliminaries}
	
	\subsection*{Higher-rank Sol-type groups}
	
	\begin{Def} \label{Def:Derivation}
		Let $\mathfrak{n}$ be a Lie algebra. A \textit{derivation} on $\mathfrak{n}$ is a linear map $D:\mathfrak{n}\to\mathfrak{n}$ such that $D([X,Y])=[D(x),Y]+[X, D(Y)]$. 
	\end{Def}
	
	We will mostly be addressing higher-rank Sol-type groups.
	
	\begin{Def} \label{Def:SolTypeGp}
		
		A \textit{Sol-type group} is a group $G=\prod_{i=1}^{n} \mathbf{N}_i \rtimes \R^k$, such that each $\mathbf{N}_i$ is a simply-connected nilpotent Lie group with Lie algebra $\mathfrak{n}_i$, $D_i$ is a derivation on $\mathfrak{n}_i$ with all eigenvalues having positive real part, and where the factor $\R^k$ acts on each $\mathbf{N}_i$ by $(\vec{v})\mathbf{N}_i(-\vec{v})=e^{\alpha_i(\vec{v})D_i}\mathbf{N}_i$ where $\alpha_i$ is a linear map from $\R^k$ to $\R$. We will denote $\prod_{i=1}^{n}\mathbf{N}_i=\mathbf{N}$, and the quotient by $\mathbf{N}$ to be $q:G\to\R^k$. We term the $\alpha_i$ to \textit{roots}.
		
		The value of $k$ is the \textit{rank} of $G$. $G$ is of \textit{higher rank} if $k\ge 2$. We will also scale the $D_i$ and $\alpha_i$ simultaneously, to assume that the smallest real part of an eigenvalue of $D_i$ is $1$.
		
	\end{Def}

	We will study the large scale geometries of higher-rank Sol-type groups with respect to certain left-invairant Riemannian metrics. Specifically, we want those metrics for which there is some lift of $\R^k$ perpendicular to $\mathbf{N}$. 
	
	\begin{Def}
		Let $G$ be a higher-rank Sol-type group, and let $g$ be a left-invariant Riemannian metric. The metric $g$ is \textit{perpendicularly split} if there is a section $s_g:\R^k\to G$ of the quotient $q:G\to \R^k$ such that $g(s_g(\R^k), \mathbf{N})=0$. We will also say that $s_g$ \textit{splits} $G$ \textit{perpendicularly}, or that $s_g$ is a \textit{perpendicular splitting} of $g$. 
	\end{Def}
	
		In general, there is always a subspace of the Lie algebra $\mathfrak{g}$ perpendicular to $\mathbf{N}$, but this subspace need not be a Lie subalgebra. A perpendicular splitting exists for a metric precisely when the subspace is a subalgebra. Note also that in rank-1, any linear subspace of $\mathfrak{g}$ is automatically a subalgebra, so that every metric has this property.

		A perpendicular splitting gives rise to a metric $s_g^* g$ on $\R^k$. We can the compare two metrics arising this way as follows.
	
	\begin{Def}
		
		 Suppose $G=\mathbf{N}\rtimes\R^k$ is a higher-rank Sol-type group, and $s_1$ and $s_2$ are perpendicular splittings for $g_1$ and $g_2$. Coordinatize $\R^k$ with any $s_1^*g_1$-orthonormal basis $\vec{e}_1, ... \vec{e}_k$, and let $\vec{v}_1, ... \vec{v}_k$ be an $s_2^*g_2$-orthonormal set. Let the matrix $A$ rewrite vectors into $\vec{v}_1, ... \vec{v}_k$-coordinates, i.e. $A_{i,  j}=g_2(\vec{v}_j, \vec{e}_i)$, and let $\langle \cdot , \cdot \rangle$ denote the dot product. Then if we have two vectors $\vec{w}_1$, $\vec{w}_2$ written in $\vec{e}_1, ... \vec{e}_k$ coordinates, $\langle \vec{w}_1, \vec{w}_2\rangle=w_1^Tw_2=g_1(\vec{w}_1, \vec{w}_2)$, while $\langle A \vec{w_1}, A\vec{w}_2\rangle=\vec{w}_1^TA^TA\vec{w}_2=g_2(\vec{w}_1, \vec{w}_2$. In this sense, the matrix $A^TA$ is the \textit{change of metric} between $s_1^*g_1$ and $s_2^*g_2$. It is evidently symmetric, so that all its singular values are in fact eigenvalues.
		
	\end{Def}
	
	We will often tacitly use this last fact and refer to the top eigenvalue as the most the change of metric can stretch a vector.
	
	A perpendicular splitting also gives rise to a coordinate system on $G=\mathbf{N} s_g(\R^k)$ by writing $h s_g(\vec{v})$ as $(h, \vec{v})$. For each $\vec{h}_i$ in $\mathfrak{n}_i=T_e(\mathbf{N}_i)$, this coordinate system gives us a non-left-invariant vector field everywhere tangent to cosets of $\mathbf{N}_i$, defined by right-translations of $\vec{h}_i$ by vectors $\vec{v}$ and left-translations by elements of $\mathbf{N}$. Equivalently, these are the coordinate pieces of $dh$, where $h:G\to \mathbf{N}$ is one half of the coordinatization above.
	
	Relative to these coordinates, a tangent vector $\vec{h}_i$ to a coset $\mathbf{N}_i\vec{v}$ has the same length as $D_eL_{-\vec{v}}\vec{h}_i$. The positivity of the eigenvalues of $D_i$ shows that this length decays asymptotically exponentially with respect to $\alpha(\vec{v})$. We will compute this more precisely in Section \ref{NonDiagonalizableDerivations}. 
	
	This allows us to define certain subspaces and paths of interest.
	
	\begin{Def}\label{HalfSpaces}
		$(G,g)$ be a higher-rank Sol-type group where $g$ splits $G$ perpendicularly. Denote by $d_i$ the path metric on $\mathbf{N}_i$ induced by the metric $g$. Let $p$ and $q$ be points in $G$. Write $p^{-1}q=( \prod_{i=1}^{n} h_i, \vec{v})$. For each $i$, there is some number $H_i$ so that, for $H\ge H_i$, $d_i(e^{-H D_i}(h_i))\le 1$. Then the \textit{i}$^{th}$ \textit{half-space} associated to $p$ and $q$, is $\alpha_i^{-1}([H_i,\infty))$. 
	\end{Def}
	
	The majority of the sequel will focus on comparing the lengths of paths that visit each half-space in turn.
	
	\begin{Def}\label{HSV}
		Retain the previous notation. A path $\gamma$ between $p$ and $q$ is said to be \textit{half-space visiting}, which we abbreviate HSV, if the image of $\gamma$ meets each half space. $\gamma$ is said to be a \textit{box path} if its derivative is tangent to left cosets of $s_g(\R^k)$ everywhere except for at most $n+1$ subintervals, along which its derivative is tangent to a coset of a different one of the $\mathbf{N}_i$. 
		
		The box path between $p$ and $q$ whose length is minimal among such paths is a \textit{box geodesic} between $p$ and $q$. The length of such a path is denoted $\rho_g(p, q)$, or just $\rho(p, q)$ if the metric is understood.
	\end{Def}
	
	\subsection*{A note on coordinates}
	
	At various times we will need to project vectors onto subspaces. If $g$ is a left-invariant Riemannian metric, and $V$ is a subspace of $\mathfrak{g}$, then $\vec{w}|_V$ denotes the orthogonal projection onto $V$. We also fix the following notation. If $s_g$ splits $(G, g)$ perpendicularly, and if $\vec{w}$ is tangent to $(s_g)_*\mathfrak{a}$, then the space $\vec{w}^\perp$ is the set of vectors tangent to $ys_g(\R^k)$ orthogonal to $\vec{w}$. We extend this notation to left-invariant vector fields.
	
	Usually, this will happen when we wish to reduce a vector, such as the derivative of a curve, into its $\mathfrak{n}$ and $(s_g)_*\mathfrak{a}$-coordinates. In Section \ref{sec:GeometryOfSplittings} we will show that the field $\nabla\alpha_i$ lies in $(s_g)_*\mathfrak{a}$. This notation will therefore allow us to split the derivative of a curve into its $\nabla\alpha_i$-coordinate and the coordinates tangent to cosets of the root kernel in a less verbose way.
	
	\subsection*{Coarse Geometry}
	
	Here we describe the equivalence relations we will be interested in for various spaces.
	
	\begin{Def}
		
		Let $(X, d_X)$ and $(Y, d_Y)$ be metric spaces. A map $f:X\to Y$ is a $(K, C)$-\textit{quasi-isometry} if for all $x_1$, $x_2$ in $X$, 
		\[\frac{1}{K}d_X(x_1, x_2)-C\le d_Y(f(x_1), f(x_2))\le Kd_X(x_1, x_2)+C,\]
		and if additionally each point in $Y$ is at most distance $C$ from a point in $f(X)$.
		
		If additionally, for each $x_1$ and $x_2$ in $X$, it holds that
		\[Kd_X(x_1, x_2)-C\le d_Y(f(x_1), f(x_2))\le Kd_X(x_1, x_2)+C,\]
		then $q$ is a $(K, C)$-\textit{rough similarity}. A rough similarity with $K=1$ is a $C$-\textit{rough isometry}. In each case, we omit the parameters if we mean to say that the map is of the desired form for some choice of parameter.
		
		A bijective map $f:X\to Y$ satisfying 
		\[\frac{1}{K}d_X(x_1, x_2)\le d_Y(f(x_1), f(x_2))\le Kd_X(x_1, x_2),\]
		is a $K$-bi-Lipschitz equivalence.
		
	\end{Def}
	
	Note that by the general theory, any two left-invariant Riemannian metrics on the same group are automatically bi-Lipschitz through the identity map. Theorem \ref{IntroRIThm} gives conditions when we may replace this multiplicative error, at arbitrary scales, with an additive error. 
	
	\subsection*{Spaces of left-invariant distances}
	
	We define a notion of distance between different distances, following \cite{OregonReyes}. 
	
	\begin{Def} \label{Def:SpaceOfMetrics}
		Let $G$ be a Lie group and by $\mathscr{D}_{Riem}(G)$ denote the set of left-invariant Riemannian distances, modulo rough similarity. For left-invariant Riemannian distances $d_1$ and $d_2$, we define $\Lambda(d_1, d_2)=\inf \{ \lambda_1\lambda_2: \text{ there is a } C>0 \text{ such that } \frac{1}{\lambda_1} d_1-C\le d_2\le \lambda_2 d_1+C \}$. If $\delta_1$ and $\delta_2$ are equivalence classes of left-invariant Riemannian distances on $G$, we define $\Delta(\delta_1, \delta_2)=\inf \{\log(\Lambda(d_1, d_2)): d_i\in \delta_i\}$.
	\end{Def}
	
	It is an exercise that $\Delta$ is a metric on $\mathscr{D}_{Riem}(G)$. If $G=\R^k$, then $\mathscr{D}_{Riem}(G)$ is isometric to the symmetric space $SL_k\R/SO_k\R$. The distance between two metrics is just the logarithm of the eccentricity of the unit ball for one relative to the other.
	
	We will be interested in those metrics that admit a perpendicular splitting.
	
	\begin{Def}
		Let $G$ be a higher-rank Sol-type Lie group. Then $\mathscr{D}_{Riem, \perp}(G)$ denotes the set of left-invariant Riemannian distances, modulo rough similarity, induced by metrics that split $G$ orthogonally.
	\end{Def}
	
	\section{Lie algebraic properties of splittings} \label{sec:LieAlgebraicProperties}
	
	Let $G=\mathbf{N}\rtimes \R^k$ be a higher-rank Sol-type Lie group. In this section we will describe properties that hold for every splitting of the quotient $q:G\to \R^n$. These will matter in particular for the perpendicular splittings with respect to a pair of metrics.
	
	For this section, $\overline{A}$ denotes an arbitrary lift of $\R^k$ to $G$. Its tangent space will be spanned by the vectors $\overline{a}_1, ... \overline{a}_k$. Together with vectors in the $\mathfrak{n}_i$, this will provide a basis for the Lie algebra $\mathfrak{g}$ on which we will do our calculations.
	
	\begin{lemma} \label{KernelsOfHeintzeQuotients}
		Let $G=\prod_{i=1}^{n}\mathbf{N}_i\rtimes \R^k$ be a higher-rank Sol-type lie group, and let $s:\R^k\to G$ be a splitting of the quotient $q:G\to \R^k$. Let $\vec{a}$ be a vector in $\R^k$ and write $s(\vec{a})=\sum_{j=1}^{n} \vec{h}_j+\sum_{j=1}^k c_j\overline{a}_j$. If $\alpha_i(\vec{a})=0$, then $\vec{h}_i=\vec{0}$.
	\end{lemma}
	
	\begin{proof}
		Let $\vec{a_1}$ be some element of $\R^k$ for which $\alpha_i(\vec{a_1})\ne 0$, and let $\sum_{j=1}^{n} \vec{h}_{j,1}+\sum_{j=1}^k c_{j,1}\overline{a}_j$ be the element corresponding to $s(\vec{a_1})$ in the $T_e(s(\R^k))$. Then since the Lie algebra is abelian, 
		\[\Big[\sum_{j=1}^{n} \vec{h}_{j,1}+\sum_{j=1}^k c_{j,1}\overline{a}_j, \sum_{j=1}^{n} \vec{h}_j+\sum_{j=1}^k c_j\overline{a}_j\Big]=0.\]
		
		The $\mathbf{n}_i$-coordinate of the above evaluates to
		\[ \alpha_i(s(\vec{a_1}))(\vec{h}_i) -\alpha_i(s(\vec{a})) \vec{h}_{j,1}+[\vec{h}_{i, 1}, \vec{h}_i].\]
		Since $\vec{a}$ is in the kernel of $\alpha_i$, we cancel to obtain
		\[\alpha_i(s(\vec{a_1}))\vec{h}_i+[\vec{h}_{i, 1}, \vec{h}_i]=0.\]
		
		Now, consider the filtration of $\mathfrak{n}_i$ by its lower central series. Let $\mathfrak{n}_i=\mathfrak{n}_{i,0}$ and $\mathfrak{n}_{i, j+1}=[\mathfrak{n}_i, \mathfrak{n}_{i, j}]$. We say an element is of step $m$ if it is in $\mathfrak{n}_m$ but not $\mathfrak{n}_{m+1}$. If $\vec{h}_i$ is nonzero and of step $m$, then $[\vec{h}_{i, 1}, \vec{h}_i]$ is of step at least $m+1$. Since $\alpha_i(s(\vec{a_1}))\ne 0$ by assumption, then we have shown two elements of different step to be equal. It follows that $\vec{h}_i$ is $0$.
	\end{proof}
	
	\begin{prop}\label{HeintzeCoordinates}
		
		Let $G=\prod_{i=1}^{n} \mathbf{N}_i\rtimes \R^k$ be a higher-rank Sol-type Lie group, and let $s:\R^k\to G$ be a splitting of the quotient $q:G\to \R^k$. Let $s(\vec{a_j})=(\vec{h_j}, \overline{\vec{a_j}})$. If $\alpha_i(\vec{a_1})=\alpha_i(\vec{a_2})$, then the $\mathbf{N}_i$ components of $\vec{h_1}$ and $\vec{h_2}$ agree.
		
	\end{prop}
	
	\begin{proof}
		
		If $\alpha_i(s(\vec{a_1}))=\alpha_i(s(\vec{a_2}))=0$, then this proposition is a direct application of the previous lemma. So suppose that $\alpha_i(s(\vec{a_1}))=\alpha_i(s(\vec{a_2}))\ne0.$
		
		Let $\sum_{j=1}^{n}\vec{h}_{j,l}+\sum_{j=1}^k c_{j,l} \overline{a_j}$ denote the lie algebra elements corresponding to $s(\vec{a_l})$. We will prove inductively that the elements  $\vec{h}_{j,1}$ and $\vec{h}_{j, 2}$ in $\mathfrak{n}_j$ agree at every step.
		
		Since $s$ is a splitting, the Lie algebra $T_e (s(\R^k))$ is abelian. As a result, the $\mathfrak{n}_i$-coordinate of the bracket $[\sum_{j=1}^{n}\vec{h}_{j,1}+\sum_{j=1}^k c_{j,1} \overline{a_j}, \sum_{j=1}^{n}\vec{h}_{j,2}+\sum_{j=1}^k c_{j,2} \overline{a_j}]$ must vanish (as must every other coordinate). Let $\mathfrak{n}_i=\bigoplus \mathfrak{n}_{i, l}$ be a grading of $\mathfrak{n}_i$, with respect to which $\vec{h}_{i, k}=\sum_l \vec{h}_{i, k, l}$. Then the $\mathfrak{n}_{i, 1}$-component of $[\sum_{j=1}^{n}\vec{h}_{j,1}+\sum_{j=1}^k c_{j,1} \overline{a_j}, \sum_{j=1}^{n}\vec{h}_{j,2}+\sum_{j=1}^k c_{j,2} \overline{a_j}]$ is $\alpha_i(\sum_{j=1}^k c_{j,1} \overline{a_j}) \vec{h}_{i, 2, 1}-\alpha_i(\sum_{j=1}^k c_{j,2} \overline{a_j}) \vec{h}_{i, 1, 1}=0$. But by assumption, $\alpha_i(\sum_{j=1}^k c_{j,1} \overline{a_j})=\alpha_i(s(\vec{a_1}))$ must equal $\alpha_i(s(\vec{a_2}))=\alpha_i(\sum_{j=1}^k c_{j,1} \overline{a_j})$. As a result, the assumption that $\alpha_i(s(\vec{a_1}))\ne 0$ implies that $\vec{h}_{i, 1, 1}=\vec{h}_{i, 2, 1}$.
		
		Suppose now that $\vec{h}_{i, 1, l}=\vec{h}_{i, 2, l}$ for $1\le l \le m$. Then by the grading of $\mathfrak{n}$, the $\mathfrak{n}_{i,m+1}$ coordinate (and below) of $[\vec{h}_{i, 1}, \vec{h}_{i, 2}]$ must vanish. We then calculate that the $\mathfrak{n}_{i,m+1}$ coordinate of 
		\[\Big[\sum_{j=1}^{n}\vec{h}_{j,1}+\sum_{j=1}^k c_{j,1} \overline{a_j}, \sum_{j=1}^{n}\vec{h}_{j,2}+\sum_{j=1}^k c_{j,2} \overline{a_j}\Big]\]
		is 
		\[\alpha_i\Big(\sum_{j=1}^k c_{j,1} \overline{a_j}\Big) \vec{h}_{i, 2, m+1}-\alpha_i\Big(\sum_{j=1}^k c_{j,2} \overline{a_j}\Big) \vec{h}_{i, 1, m+1},\]
		
		which must vanish. Therefore $\vec{h}_{i, 1, m+1}=\vec{h}_{i, 2, m+1}$.    
	\end{proof}
		
	\section{Geometric properties of splittings} \label{sec:GeometryOfSplittings}
	
	We now introduce a left-invariant Riemannian metric $g$ for which $s_g:\R^k\to G$ is a perpendicular splitting of the quotient $q:G\to \R^k$. The goal of this section is to establish that, for any other splitting $s_1:\R^k\to G$, there is an element of $x\in\prod_{i=1}^{n} \mathbf{N}_i$ so that the functions $xs_1$ and $s_g$ are at finite supremum distance. That is, if $d_g$ is the distance associated to $g$, then $\sup_{\vec{v}\in\R^k} d_g(xs_1(\vec{v}),s_g(\vec{v}))$ is finite. The same will then hold for any other left-invariant Riemannian distance, since each one is bi-Lipschitz to $d_g$. 
	
	We start by observing that perpendicular splittings are geodesic.
	
	\begin{lemma}
		Let $(G, g)$ split perpendicularly. Then $s_g(\R^k)$ is totally geodesic.
	\end{lemma}
	
	\begin{proof}
		We must verify that for every vector $\vec{v}\in \R^k$, $\nabla_{s_g(\vec{v})}s_g(\vec{v})=0$ where $\nabla$ denotes the Levi-Civita connection for $g$. By Koszul's formula,
		\[\nabla_{s_g(\vec{v})}s_g(\vec{v}) = \frac{1}{2}(\text{ad}_{s_g(\vec{v})} s_g(\vec{v}) -2\text{ad}^*_{s_g(\vec{v})} s_g(\vec{v})),\]
		
		where $\text{ad}^*_XY$ is the vector field $Z$ such that $g(Z, W)=g(Y, \text{ad}_X Z)$. Since $\text{ad}$ is antisymmetric, $\text{ad}_{s_g(\vec{v})} s_g(\vec{v})$ is $0$. Since $\text{ad}$ is valued in $\prod_{i=1}^{n}\mathbf{N}_i$, which is perpendicular to $s_g(\R^k)$, it follows that $\text{ad}^*_{s_g(\vec{v})} s_g(\vec{v})$ is the vector that is $0$ when paired with any other vector. 
	\end{proof}
	
	Since $s_g(\R^k)$ is totally geodesic, we can immediately describe the direction of fastest increase of the maps $\alpha_i$.
	
	\begin{lemma}
		
		Let $G=\prod_{i=1}^{n} \mathbf{N}_i\rtimes \R^k$ be a higher-rank Sol-type lie group, with a left-invariant Riemannian metric $g$ splitting $G$ perpendicularly. The field $\nabla(\alpha_i\circ q)$ is left-invariant and lies in $T_e(s_g(\R^k))$.    
	\end{lemma}
	
	\begin{proof}
		
		We first prove left-invariance. Let $U$ be a small open neighborhood of the identity, and let $(h, a)\in G$. Then $(h,a)U$ is a small open neighborhood of $(h,a)$. Since $\alpha_i\circ q$ depends only on the $\R^k$-coordinate of elements, and left multiplying by $(h,a)$ acts as a translation by $a$ on the $\R^k$-coordinate of $U$, $\alpha_i\circ q \circ L_{(h,a)}=\alpha_i\circ q+\alpha_i(a)$. Hence $D_e L_{h, a} \nabla(\alpha_1\circ q)(e)=\nabla(\alpha_1\circ q)(h,a)$ as required.
		
		Suppose $\vec{v}$ is a tangent vector at the origin, and write $\vec{v}$ as a sum of vectors $\vec{v}=\sum_{i=1}^{n} \vec{h}_i+\sum_{i=1}^k c_is_g(\vec{a}_i)$. Since $s_g$, where each $\vec{h}_i$ is in $\mathbf{n}_i$. Since $s_g$ is an orthogonal splitting for $g$, 
		\[\lVert\vec{v}\rVert=\sqrt{ \lVert \sum_{i=1}^{n} \vec{h}_i\rVert^2 + \lVert\sum_{i=1}^k c_is_g(\vec{a}_i)\rVert ^2}.\]
		
		If any of the $\vec{h}_i$ are nonzero, then $\lVert\vec{v}\rVert>\lVert\sum_{i=1}^n c_is_g(\vec{a}_i)\rVert$. But this is then a shorter vector than $\vec{v}$ so that $D_e \nabla\alpha_i\circ q (\sum_{i=1}^n c_is_g(\vec{a}_i))$ is the same as $D_e \nabla\alpha_i\circ q (\vec{v})$, so that $\vec{v}$ is not the gradient direction.
	\end{proof}
	
	Henceforth, when we have a metric for which $G$ splits perpendicularly, we will simply refer to this gradient field as $\nabla\alpha_i$, suppressing both the quotient $q$ and the dependence on the metric.
	
	\begin{Def}
		Let $G$ be a higher-rank Sol-type group and $g$ a left-invariant metric splitting $G$ perpendicularly. For each $1\le i\le n$, we define group quotient $\pi_i$ as follows. Extend $\nabla\alpha_i$ to an orthogonal basis of $s_g(\R^k)$, and call this set of vectors (not including $\nabla\alpha_i$) $S_i$. Then the quotient $\pi_i$ has kernel spanned by $S_i$ and $\prod_{j\ne i} \mathbf{N}_i$. The image of $\pi_i$ is a Heintze group $\mathbf{N}_i\rtimes \R$, which we denote $\mathbf{H}_i$. We always assign $\mathbf{H}_i$ a metric keeping the lengths of and angles between the vectors $\nabla\alpha_i$ and $\mathfrak{n}_i$ the same in the quotient.
		
		Since the spaces $\mathbf{N}_i$ are not perpendicular, we choose $L$ so that the inclusion of $\mathbf{H}_i$ or any of its cosets into $G$ is $L$-Lipschitz.
	\end{Def}
	
	We next show that splittings (of any type, perpendicular or otherwise), are parametrized by $k$-tuples in $\prod_{i=1}^{n} \mathbf{N}_i$.
	
	\begin{lemma}
		Let $(G, g)$ split perpendicularly, and let $s$ be any splitting of the quotient $q:G\to\R^k$. Then for each $i$, $\pi_i(s(\R^k))$ is a 1-parameter subgroup of $\mathbf{H}_i$ that does not sit inside $\mathbf{N}_i$, and is therefore determined by two limit points, one of which is the point at infinity, and the other is in the copy of $\mathbf{N}_i$ in the boundary.
	\end{lemma}
	
	\begin{proof}
		By Lemma \ref{KernelsOfHeintzeQuotients} and Proposition \ref{HeintzeCoordinates}, $\pi_i(s(\R^k))$ is defined by 1 parameter. It is a subgroup because it is a quotient of a group. The remainder follows because a 1-parameter subgroup in $\mathbf{H}_i$ that does not lie in $\mathbf{N}_i$ is asymptotic to the point at infinity and one other point.
	\end{proof}
	
	\begin{prop}
		Let $(G, g)$ split perpendicularly, and let $s_1$ be any other splitting of the quotient $q:G\to\R^k$ besides $s_g$. Then there is a constant $C$ and an element $x$ in $\prod_{i=1}^{n}\mathbf{N}_i$ so that $d_g(xs_1(\vec{v}), s_g(\vec{v}))\le C$ for any vector $\vec{v}$ in $\R^k$.
	\end{prop}
	
	\begin{proof}
		Denote by $a_i$ the magnitude $|\nabla\alpha_i|$.
		
		By the previous lemma, let $n_{i,1}$ be the element of $\mathbf{N}_i$ that $\pi_i s_1(\R^k)$ is asymptotic to, and similarly $n_{i}$ the element that $\pi_1(s_g(\R^k))$ is asymptotic to. We set $x=\prod_{i=1}^{n} n_{i, 1}n_{i}^{-1}$.
		
		$\pi_i(xs_1(\R^k))$ and $\pi_i(s_g(\R^k))$ are quasi-geodesics in $\mathbf{H}_i$ which have the same endpoints at infinity. Since $\mathbf{H}_i$ is negatively curved, it follows that there is some $C_i$ so that the two images have Hausdorff distance no more than $C_i$. Moreover, since the derivative of $\pi_i(s_g(\R^k))$ at the identity is perpendicular to the horocyclic copy of $\mathbf{N}_i$, it follows that $s_g(\R^k)$ is a minimizing geodesic. Since points at $\nabla\alpha_i$-coordinate greater than $C_i$ apart cannot be within $C_i$ of one another, it follows that for any vector $\vec{v}$ in $\R^k$, $\pi_1(xs_1(\vec{v}))$ is at distance $C_1$ from a point on $\pi_1(s_g(\R^n))$ at height in $[a_i(\alpha_i(\vec{v})-C_i),a_i(\alpha_i(\vec{v})+C_i)]$, and thus no more than distance $2C_i$ from $\pi_1(s_g(\vec{v}))$, which is necessarily at height $\alpha_i(\vec{v})$.
		
		So the points $xs_1(\vec{v})$ and $s_g(\vec{v})$ are at distance at most $2C_i$ in the $i^{\text{th}}$ projection each projection. Now, since the $\mathbf{N}_i$ all commute, if $S$ is a set of indices not including $i$, then $\pi_i \prod_S n_{j, 1}n_{j}^{-1} s_1=\pi_i s_1$. Therefore, $\pi_j\big(\prod _{i=1}^{j-1} n_{i, 1}n_{i}^{-1} s_1(\vec{v})\big)=\pi_j(s_1(\vec{v}))$ while $\pi_j\big(\prod _{i=1}^{j} n_{i, 1}n_{i}^{-1} s_1(\vec{v})\big))=\pi_j(s_g(\vec{v}))$. Since these two points lie on a left coset of $\mathbf{H}_j$, they are at most $2LC_j$ apart in the metric $d_g$. It follows by the triangle inequality that 
		\begin{align*}
			d(xs_1(\vec{v}), s_g(\vec{v})) & \le \sum_{j=1}^{n} d\Big(\big(\prod_{l=1}^j n_{l, 1}n_{l}^{-1} \big)s_1 (\vec{v}), \prod_{l=1}^{j-1} n_{l, 1}n_{l}^{-1} s_1 (\vec{v})\Big)\\
			&\le \sum_{j=1}^{n} 2L C_j.
		\end{align*}  
	\end{proof}
	
	We may immediately apply this result to two different splittings $s_1$ and $s_2$, neither of which is necessarily $s_g$ by simply summing constants. Moreover, this also applies for any two cosets of $s_1(\R^k)$ and $s_2(\R^k)$.
	
	\begin{cor} \label{CloseLiftCosets}
		Let $(G,g)$ split perpendicularly, and let $s_1$ and $s_2$ be any two splittings of the quotient. There is a constant $C$ depending on $s_1$ and $s_2$ and an element $x$ in $\prod_{i=1}^{n}\mathbf{N}_i$ so that if $\vec{v}$ is any element of $\R^k$ and $x_1$ and $x_2$ are any elements of $G$, then $d((x_2xx_1^{-1})x_1s_1(\vec{v}), x_2s_2(\vec{v}))\le C$. \qed
	\end{cor}
	
	One final topic of geometric interest is the distortion of the nilpotent subgroups $\mathbf{N}_i$. We will wish to show that, if a path stays far from the defining half-space $H_i$, then a large amount of distance must be spent moving along the $\mathfrak{n}_i$-directions.
		
	\begin{Def}
		Let $G$ be a higher-rank Sol-type group, with left-invariant Riemannian metric $g$ and let $s$ be any (not necessarily perpendicular) splitting. Denote the Lie algebra $\mathfrak{g}=\prod_{i=1}^n\mathfrak{n}_i\rtimes\R^k$ as before. We will say that $\mathbf{N}_i$ is \textit{eventually exponentially distorted} if there exist positive constants $T$ and $C$, and $a>1$, so that if $\alpha_i(\vec{v})\ge T$, then for all $\vec{h}$ in $\mathfrak{n}_i$, 
		\[Ca^{\alpha_i(\vec{v})} \lVert\vec{h} \rVert_g\le \lVert R_{s(-\vec{v})}(\vec{h})\rVert_g,\]
		where $R_{x}$ denotes right-multiplication by $x$.
	\end{Def}
	
	Note that by the left-invariance of $g$ and the fact that left and right multiplications commute, if $\mathbf{N}_i$ is eventually exponentially distorted, we can make the same statement for each of its cosets.
	
	As a warm-up for the general case, we will show that when the derivation $D_i$ is diagonalizable, then $\mathbf{N}_i$ is eventually exponentially distorted.
	
	\begin{lemma} \label{PathsThatAvoidHoroballs}
		
		Let $G=\prod_{i=1}^n \mathbf{N}_i \rtimes\R^k$ be a higher-rank Sol-type group, and let $g$ be a left invariant Riemannian metric on $G$ with $s$ any splitting. Let the derivation $D_i$ be diagonalizable over the real numbers, with each eigenvalue positive as before, and the smallest normalized to $1$. There are constants $C>0$ and $a>1$ so that for each $\vec{v}$ in $\R^k$ and $h$ in $\mathfrak{n}_i$,
		\[ Ca^{\alpha_i(\vec{v})}\lVert \vec{h}\rVert_g \le \lVert R_{s(-\vec{v})}(\vec{h})\rVert_g. \]
		
	\end{lemma}
	
	That is, when $D_i$ is diagonalizable, $\mathbf{N}_i$ is immediately exponentially distorted (eventually exponentially distorted with $t=0$).
	
	\begin{proof}
		
		Let $g_i$ be a metric arising from $g|_{\mathfrak{n}_i}$ by making the $D_i$-eigenbasis of $\mathfrak{n}_i$ orthogonal. Then the desired inequality holds with $C=1$ and $a=e$ for $g_i$. Let the change of metric from $g|_{\mathfrak{n}_i}$ to $g_i$ have top and bottom eigenvalues $\lambda_{1}$ and $\lambda_{\dim\mathfrak{n}_i}$. 
		
		\begin{align*}
			\lVert R_{s(-\vec{v})}(\vec{h})\rVert_g &\ge \frac{1}{\lambda_1} \lVert R_{s(-\vec{v})}(\vec{h})\rVert_{g_i}\\
			&\ge \frac{1}{\lambda_1} e^{\alpha_i(\vec{v})}\lVert \vec{h}\rVert_{g_i}\\
			&\ge \frac{\lambda_{ \dim\mathfrak{n}_i}}{\lambda_1} e^{\alpha_i(\vec{v})}\lVert \vec{h}\rVert_{g}.
		\end{align*}
		
		The Lemma then holds for $C=\frac{\lambda_{ \dim\mathfrak{n}_i}}{\lambda_1}$ and $a=e.$
	\end{proof}
	
	So far we have not used the parameter $T$ in the definition of eventual exponential distortion. This parameter will be necessary when one of the derivations cannot be diagonalizable. Our analysis of the on $\mathfrak{n}_i$ in this case follows Peng in \cite{Peng1, Peng2}.
	
	\begin{Def}
		
		Let $\mathfrak{n}$ be a nilpotent Lie algebra, with $D:\mathfrak{n}\to \mathfrak{n}$ a derivation. Choose some basis for which $D$ is in real Jordan form, and decompose $D$ into a sum $D=\delta+\nu+\sigma$, where $\delta$ is diagonal, $\nu$ is strictly upper triangular, and $\sigma$ is antisymmetric. Then the \textit{Absolute Jordan Form} of $D$, denoted $|D|$, is $\delta+\nu$.
		
	\end{Def}
	
	We first show that, at the level of Lie groups, $D$ and $|D|$ give rise to bi-Lipschitz equivalent Heintze groups.
	
	\begin{lemma}
		
		Let $N$ be a simply-connected nilpotent Lie group with Lie algebra $\mathfrak{n}$, and $D:\mathfrak{n}\to \mathfrak{n}$ a derivation. Let $G=N_D\rtimes\R$, and $|G|=N_{|D|}\rtimes\R$, with left-invariant Riemannian metrics $g$ and $|g|$ on $G$ and $|G|$. Then there is an $L$ depending only on $g$ so that $G$ and $|G|$ are $L$ bi-Lipschitz.
		
	\end{lemma}
	
	\begin{proof}
		
		Since any pair of left-invariant Riemannian metrics induce bi-Lipschitz equivalent distances, we may as well take $|g|$ to agree with $g$ on $N$. Further, let $A$ and $|A|$ be lifts of $\R$ to $G$ and $|G|$ perpendicular to $N$ with respect to $g$ and $|g|$ respectively (these always exist, see \cite{LDPX} Lemma 3.1). Then if we normalize the copy of $\R$ in $G$ by its $g$-length, the positive element of $g$-length $l$ acts by $e^{alD}$ on $N$ under conjugation, where $a$ depends only on $g$. We choose the restriction of $|g|$ to $|A|$ so that the positive element of $|A|$ of $|g|$-length $l$ acts by $e^{al|D|}$ on $N$.
		
		The image of $e^{t\sigma}$ is a compact group. Thus $e^{alD}$ and $e^{al|D|}$ are $L$ bi-Lipschitz operators on $N$. If we give $G$ and $|G|$ the coordinates $(n, t)$ where $t$ is taken to lie in $A$ for $G$ and $|A|$ for $|G|$, then the map from $G$ to $|G|$ that sends each coordinate pair in $G$ to the matching coordinate pair in $|G|$ is $L$ bi-Lipschitz: it preserves the length of the $\R$ factor, and changes the lengths of vectors tangent to right cosets of $N$ in an $L$ bi-Lipschitz way. Since these spaces are orthogonal, this suffices.
	\end{proof}
	
	We next wish to understand the metric on the group $|G|$ in the case that all the eigenvalues of $D$ have positive real part. As described in \cite{Peng1}, in restriction to each Jordan block $V$,
	\[e^{t|D|}|_V=e^{|\delta|_Vt}\begin{bmatrix}
		1 & |\nu|_Vt & \frac{(|\nu|_Vt)^2}{2} & \dots & \frac{(|\nu|_Vt)^n}{n!}\\ 0 & 1 & |\nu|_Vt & \dots & \frac{(|\nu|_Vt)^{n-1}}{(n-1)!}\\
		& & \ddots & & \vdots \\ & & & 1 & |\nu|_Vt \\
		& & & & 1
	\end{bmatrix},\]
	where $|\delta|_V$ and $|\nu|_V$ are the rates of change of $\delta$ and $\nu$ with respect to length on the Jordan block $V$. Working block-by-block, we will get a lower bound good enough to show that if $G$ is a higher-rank Sol-type group, and one of its defining pairs $(\mathbf{N}_i, D_i)$ is $(N, D)$, then $\mathbf{N}_i$ is eventually exponentially distorted.
	
	\begin{prop}
		Let $D$ be a derivation on the nilpotent Lie algebra $\mathfrak{n}$ all of whose eigenvalues have positive real parts, and let $\mathbf{N}$ the simply connected nilpotent Lie group corresponding to $\mathfrak{n}$. Let $|G|=\mathbf{N}_{|D|}\rtimes\R$ with left-invariant Riemannian metric $g$, and let $\vec{h}$ be tangent to $\mathfrak{n}$. Denote by $|\delta|$ the minimum $|\delta_V|$ on any Jordan block $V$. There are positive real numbers $C$ and $T$ so that if $t>T$, then \[Ce^{\frac{|\delta|t}{2}}\lVert\vec{h}\rVert\le \lVert R_{-t}\vec{h}\rVert,\]
		where $-t$ here refers to the lift of $t\in \R$ to $|G|$.
	\end{prop}

	\begin{proof}
		
		We will prove that $\lVert R_{t}\vec{h}\rVert\le C' e^{\frac{-|\delta|t}{2}}\lVert \vec{h}\rVert$, for $C'$ some positive constant. This will show the result for the vector $R_{-t}\vec{h}$. Since right-translation is a permutation, the result will then hold for each vector $\vec{h}$.
		
		Choose a basis $(\vec{v}_1, ... \vec{v}_m)$ for $\mathfrak{n}$ for which $|D|$ is in Jordan form. Let $L_1$ be the bi-Lipschitz constant comparing the norms $\lVert\cdot\rVert_{g|_{\mathbf{N}}}$ and $L^1(\vec{v})=\sum_{i=1}^m |v_i|$. 
		
		The length of $R_{t}\vec{h}$ is the same as that of $L_{-t}R_{t}\vec{h}$, which is $e^{-t|D|}(\vec{h})$. That is, right-multiplying by $t$ acts on $\vec{h}$ by the inverse of the matrix above. Since the inverse of a matrix is a rational function in the entries of the matrix, we determine that each entry in the matrix is $e^{-|\delta|_V t}P_{i, j}(|\nu|_Vt)$, where $P_{i, j}$ is a rational function and $V$ is the Jordan block containing entry ${i, j}$. 
		
		Since there are only finitely many of these, we can take $T$ large enough that for $t>T$, 
		\[e^{-|\delta|_V t}|P_{i, j}(|\nu|_Vt)|\le e^{-\frac{|\delta|t}{2}}.\]
		Thus each entry in the $R_{t} \vec{h}$ is of size at most $me^{-\frac{|\delta|t}{2}}L^1(\vec{h})$. Hence $L^1(R_t \vec{h})$ is at most $m^2e^{-\frac{|\delta|t}{2}}L^1(\vec{h})$. We therefore calculate
		
		\begin{align*}
			\lVert R_t\vec{h}\rVert_g&\le L_1L^1(R_t\vec{h})\\
			&\le m^2L_1 e^{-\frac{|\delta|t}{2}}L^1(\vec{h})\\
			&\le m^2L_1^2e^{-\frac{|\delta|t}{2}} \lVert\vec{h}\rVert_g,
		\end{align*}
		
		as desired.
	\end{proof}
	
	Since the Heinze group $|G|$ above is bi-Lipschitz to the Heinze group $G$ defined by the derivation $D$, we obtain the same result for $G$ with a different constant $C$.
	
	\begin{cor} \label{SquareRootScalingSpeed}
		Let $D$ be a derivation on a Lie algebra $\mathbf{n}$ with all eigenvalues of positive real part. Let $G=\mathbf{N}_D\rtimes \R$, with $g$ a left-invariant Riemannian metric on $G$. There  are positive real numbers $C$ and $T$ so that, for $t>T$, \[Ce^{\frac{|\delta|t}{2}}\lVert\vec{h}\rVert\le \lVert R_{-t}\vec{h}\rVert.\] \qed
	\end{cor}
	%\begin{proof}
		%Let $|G|$ be the related group as before, and take $T$ large enough so that for $t>T$, $e^{-|\delta|_V t}|P_{i, j}(|\nu|_Vt)|\le \frac{e^{-\frac{|\delta|t}{2}}}{LL_1^2m^2}$, where $L$ is the bi-Lipschitz constant between $|G|$ and $G$. Then $\lVert(\vec{v},0)\rVert_{|g|}\le \frac{1}{L}e^{\frac{-|\delta|t}{2}}\lVert\vec{v}\rVert_{|g||_\mathbf{N}}$ by the previous proposition, where $|g||_\mathbf{N}$ denotes the restriction of $|g|$ to $\mathbf{N}$. However, $|g||_{\mathbf{N}}=g|_{\mathbf{N}}$, so that $\lVert(\vec{v},0)\rVert_{|g|}\le \frac{1}{L}e^{\frac{-|\delta|t}{2}}\lVert\vec{v}\rVert_{g|_\mathbf{N}}$. The $L$ bi-Lipschitz equivalence between $\lVert(\vec{v}, 0)\rVert_{|g|}$ and $\lVert(\vec{v}, 0)\rVert_{g}$ provides \[\lVert(\vec{v},0)\rVert_{g}\le e^{\frac{-|\delta|t}{2}}\lVert\vec{v}\rVert_{g|_\mathbf{N}}\] as required.
	%\end{proof}
	
	As a consequence, each nilpotent group $\mathbf{N}_i$ in a higher-rank Sol-type group is eventually exponentially distorted.
	
	\begin{prop} \label{AllNilpotentDirectionsAreEventuallyExponentiallyDistorted}
		Let $G=\prod_{i=1}^n \mathbf{N}_i \rtimes\R^k$ be a higher-rank Sol-type group, and let $g$ be a left invariant Riemannian metric on $G$ with $s$ any splitting. Let the derivation $D_i$ have positive real part of each of its eigenvalues, the smallest of which is normalized to be $1$. Then $\mathbf{N}_i$ is eventually exponentially distorted.
	\end{prop}
	
	\begin{proof}
		If $D_i$ is diagonalizable, this is an immediate application of Lemma \ref{PathsThatAvoidHoroballs}.
		
		If $D_i$ is not diagonalizable, then there is a quotient $q_i$ from $G$ to the Heintze group $\mathbf{N}_i\rtimes\R$ whose defining derivation is again $D_i$. Endow $\mathbf{N}_i\rtimes\R$ with a metric $g'$ agreeing with $g$ in restriction to $\mathfrak{n}_i$. Since $q_i$ is a Lie group homomorphism, it is $L$-Lipchitz for some $L>0$. Moreover, one calculates directly that $q_i\circ s(\vec{v})=\alpha_i(\vec{v})$.
		
		 As a consequence, if a vector $\vec{v}$ has $\alpha_i(\vec{v})>T$, where $T$ is the value in Corollary \ref{SquareRootScalingSpeed}, then $q_i (R_{s(\vec{v})}\vec{h})$ satisfies the conditions of Corollary \ref{SquareRootScalingSpeed}. Therefore, for any $\vec{h}$ in $\mathfrak{n}_i$,
		 
		 \begin{align*}
		 	\lVert R_{s(-\vec{v})}\vec{h} \rVert_g &\ge L \lVert q_i\circ R_{s(-\vec{v})}\vec{h} \rVert_{g'}\\
		 	&\ge L \lVert R_{-\alpha_i(\vec{v})} (q_i)_*\vec{h}\rVert\\
		 	&\ge LC e^{\frac{\alpha_i(\vec{v})}{2}}\lVert\vec{h}\rVert_{g'}\\
		 	&\ge LC e^{\frac{\alpha_i(\vec{v})}{2}}\lVert\vec{h}\rVert_{g},
		 \end{align*}
		 where in the final step, we used the fact that, $\vec{h}$ and $(q_i)_*\vec{h}$ are tangent to the respective copies of $\mathbf{N}_i$, and $g$ and $g'$ agree in restriction to the respective copies of $\mathfrak{n}_i$.
	\end{proof}
	
	Since any given higher-rank Sol-type group has only finitely many nilpotent factors $\mathbf{N}_i$, each eventually exponentially distorted with different values of the parameters $C$, $T$, and $a$, we can take these values to be uniform. In addition, in the sequel, we will care about how far a path is from visiting the half spaces $H_i$, rather than the difference in $\alpha_i$-values per se. We summarize these two rephrasings in the following corollary.
	
	\begin{cor}\label{UniformEventualExponentialDistortion}
		Let $G=\prod_{i=1}^n\mathbf{N}_i\rtimes \R^k$ be a higher-rank Sol-type group with $g$ a left-invariant Riemannian metric and $s_g$ a perpendicular splitting. Give $G$ coordinates with $s_g(\R^k)$-factor on the right. There are constants $C>0$, $T>0$, and $a>1$ so that for all $t$ and all $1\le i \le n$, the following holds. If $H$ is a half-space for $\alpha_i$, and $\vec{h}$ is a vector tangent to $\mathfrak{n}_i$ based at a point a distance $t$ from $H$, and $\vec{h}'$ is the vector equal in coordinates based at a point in $H$ then 
		\[ Ca^t \lVert\vec{h}' \rVert \le \lVert\vec{h} \rVert\]
	\end{cor}
	
	\begin{proof}
		
		Denote $\lVert\alpha_i\rVert=\lVert\nabla\alpha_i\rVert$. Also, denote $C_i$, $T_i$, and $a_i$ the values of $C$, $T$, and $a$ for $\mathbf{N}_i$ guaranteed by Proposition \ref{AllNilpotentDirectionsAreEventuallyExponentiallyDistorted}. 
		
		Then $\vec{h}$ differs from a point on the same left $\mathbf{N}_i$-coset as $\vec{h}'$ by right-multiplication by a vector $\vec{v}\in s_g(\R^k)$ so that $\alpha_i(\vec{v})=\lVert\alpha_i\rVert t$. Since lengths are left-invariant, we may as well take $\vec{h}'$ to be $R_{\vec{v}}\vec{h}$. In order to apply Proposition \ref{AllNilpotentDirectionsAreEventuallyExponentiallyDistorted} to $R_{s_g(-\vec{v}}\vec{h}'=\vec{h}$, we need $t\lVert\alpha_i\rVert\ge T_i$, i.e. $t\ge \frac{T_i}{\lVert\alpha_i\rVert}$. As long as $t\ge \frac{T_i}{\lVert\alpha_i\rVert}$, then 
		\[ C_ia_i^{\alpha_i(\vec{v})}\lVert \vec{h}'\rVert \le \lVert \vec{h}\rVert.\]
		Again, since $\alpha_i(\vec{v})=t\lVert\alpha_i\rVert t$, we obtain
		\[ C_i\Big{(}a_i^{\lVert\alpha_i\rVert}\Big{)}^t\lVert \vec{h}'\rVert \le \lVert \vec{h}\rVert.\]
		
		The result then holds for $T=\max_i \frac{T_i}{\lVert\alpha_i\rVert}$, $C=\min_i C_i$, and $a=\min_i a_i^{\lVert\alpha_i\rVert}$.
	\end{proof}
	
	\section{Comparing lengths of box geodesics} \label{sec:BoxGeodesicLengths}
	
	In this section we define, for each left-invariant Riemannian metric $g$ splitting $G$ perpendicularly, a distance approximation $\rho_g$. We then show, for a pair of left-invariant Riemannian metrics $g_1$ and $g_2$, that the distance $d_2$ for $g_2$ is bounded above in terms of the approximate metric $\rho_{g_1}$, which we will abbreviate $\rho_1$.
	
	\begin{Def}
		
		Let $G=\prod_{i=1}^n \mathbf{N}_i\rtimes\R^k$ be a higher-rank Sol-type group, and $g$ a left-invariant metric splitting $G$ perpendicularly. Let $p$ and $q$ be any two points in $G$, and let $S(p, q)$ be defined as the set \newline $\{\gamma| \gamma \text{ is a path between } p \text{ and } q, \gamma \text{ is half-space visiting box path} \}$. We define $\rho_g(p, q) = \inf_{S(p,q)} \ell_g(\gamma)$. A path $\eta$ achieving the value of $\rho_g$ will be termed a \textit{box geodesic}, and $\rho_g$ is called the \textit{box-geodesic distance} associated to $g$. We will sometimes simply refer to $\rho$ when there is only one left-invariant Riemannian metric at hand.
		
	\end{Def}
	
	It turns out that $\rho_g$ differs from the distance $d$ induced by $g$ by a constant $C$ depending on $g$. That is, box geodesics are uniformly roughly geodesic, where the implied constant depends on the underlying metric.
	
	\begin{thm} \label{BoxGeodesicsAreRoughGeodesics}
		Let $G$ be a higher-rank Sol-type group, with left-invariant Riemannian metric $g$ splitting $G$ perpendicularly. There is a $C$ such that $|d(p, q), \rho_g(p, q)|\le C$ for all $p$ and $q$.
	\end{thm}
	
	The proof of this theorem is rather involved, and we will postpone it for now. In this section we will demonstrate that Theorem \ref{BoxGeodesicsAreRoughGeodesics} implies the main theorem.
	
	\begin{prop} \label{EigenvalueDistanceUpperBoundProp}
		
		Let $G$ be a higher-rank Sol-type group, and let $g_1$ and $g_2$ be two left-invariant Riemannian metrics with perpendicular splittings $s_1$ and $s_2$. Let $d_2$ (resp. $\rho_1$) be the geodesic (resp. box-geodesic) distance associated to $g_2$ (resp. $g_1$). Let the top eigenvalue of the change of metric from $s_1^*g_1$ to $s_2^*g_2$ on $\R^k$ be $\lambda_1$. Then there is a $C$ so that for all $p$, $q$ in $G$,
		\[ d_2(p, q)\le \lambda_1\rho_1(p, q)+C.\]
		
	\end{prop}
	
	\begin{proof}
		
		Let $d_1$ be the distance for $g_1$, and let $L$ be the constant for which $d_1$ and $d_2$ are $L$ bi-Lipschitz. Denote by $C_1$ be the constant in Corollary \ref{CloseLiftCosets} applied to the metric $d_2$ and the splittings $s_1$, $s_2$. Let $x\in \prod_{i=1}^{n} \mathbf{N}_i$ so that for $n\in \prod_{i=1}^{n} \mathbf{N}_i$, $n s_1(\vec{v})$ is within $d_1$-distance $C_1$ of $nxs_2(\vec{v})$. 
		
		Suppose $\eta_1$ is a box geodesic for $g_1$. As before $\eta_1$ consists of at most $n+1$ segments $\eta_{i,1}$ in cosets $n_is_1(\R^k)$ of $s_1(\R^k)$, and at most $n$ segments $\xi_i$ in cosets of the $\mathbf{N}_i$ of length at most $1$. For notational convenience, we will parameterize each such path by $[0,1]$, and we will add extra constant paths in if necessary so that $\eta_1$ is always the concatenation $\eta_{1, 1}\xi_1\eta_{2, 1}\xi_2...\eta_{n+1, 1}$. 
		
		For each of the $\eta_{i, 1}$, denote by $\eta_{i,2}$ the segment $n_ixs_2(q(\eta_{i, 1}(t)))$. Then we see immediately that
		\[ \ell_2(\eta_{i, 2})\le \lambda_1 \ell_1(\eta_{i, 1}),\]
		
		because the two are lifts of the same path in $\R^k$, and
		\[d_2(\xi_{i,1}(t), \xi_{i, 2}(t))\le C_1,\]
		
		since the two points are at $d_1$-distance no more than $C_1$. Additionally, since $\eta_{i, 1}(1)$ and $\eta_{i+1, 1}(0)$ are separated by the segment $\xi_i$ of $d_1$-length at most $1$, it follows that 
		\[ d_2(\eta_{i, 1}(1),\eta_{i+1, 1}(0))\le L.\]
		
		Now, $\eta_{1, 1}(0)=p$ and $\eta_{n+2, 1}(1)=q$. Therefore, by the triangle inequality,
		\begin{align*}
			d_2(p,q) & \le d_2(\eta_{1, 1}(0), \eta_{1, 2}(0))+\ell_2(\eta_{1,2})+d_2(\eta_{1, 2}(1), \eta_{2,2}(0)) + ... + d_2(\eta_{n+1, 2}(1), \eta_{n+1, 1}(1))\\
			&= \sum_{i=1}^{n+1} \ell_2 (\eta_{i, 2})+\sum_{i=1}^{n} \big(d_2(\eta_{i, 2}(1), \eta_{i+1,2}(0))\big) +d_2(\eta_{1, 1}(0), \eta_{1, 2}(0)) +  d_2(\eta_{n+1, 2}(1), \eta_{n+1, 1}(1))\\
			&\le \sum_{i=1}^{n+1} \lambda_1\ell_1 (\eta_{i, 1})+\sum_{i=1}^{n} d_2(\eta_{i, 2}(1), \eta_{i+1,2}(0)) +2C_1\\
			&\le \lambda_1\rho_1(p,q) +\sum_{i=1}^{n} d_2(\eta_{i, 2}(1), \eta_{i+1,2}(0)) +2C_1\\
			&\le \lambda_1\rho_1(p,q) +\sum_{i=1}^{n} \big[d_2(\eta_{i, 2}(1), \eta_{i,1}(1)) + d_2(\eta_{i, 1}(1), \eta_{i+1,1}(0))+ d_2(\eta_{i+1, 1}(0), \eta_{i+1,2}(0))\big] +2LC_1\\
			&\le\lambda_1\rho_1(p,q) + 2nC_1+nL+2C_1.
		\end{align*}
		
		The Proposition is proven with $C=2(n+1)C_1+nL$.
	\end{proof}
	
	If $\lambda_k$ is the smallest eigenvalue of the change of metric from $s_1^*g_1$ to $s_2^*g_2$, then $\frac{1}{\lambda_k}$ is the largest eigenvalue of the reverse change of metric. As a result, reversing the roles of $g_1$ and $g_2$, and increasing the value of $C$ as necessary, we obtain 
	\[ d_1(p, q)\le \frac{1}{\lambda_k} \rho_2(p,q)+C,\]
	so that
	\[ \lambda_k d_1(p, q)-C\le \rho_2(p,q).\]
	As a consequence, Theorem \ref{BoxGeodesicsAreRoughGeodesics} will imply Theorems \ref{EigenvaluesAndDistancesIntro} and \ref{SpaceOfMetricStructures}.
	
	\section{Surgery in Euclidean Space} \label{EuclideanSurgery}
	
	In order to bound the difference between $\rho_g$ and $d_g$, the box geodesic distance and geodesic distance for a metric $g$, we will need to modify a geodesic path $\gamma$ to make it half-space visiting, and control the extra length this creates. In this section, we lay out terminology to describe the changes we will be making to the curve $q(\gamma)$. The definitions and lemmas in this section are included as separate concepts not for their mathematical depth, but because they provide an outline for the proof of Theorem \ref{BoxGeodesicsAreRoughGeodesics}. 
	
	Note that we have little control over the shape of $q(\gamma)$ based on the fact that $\gamma$ is geodesic, so all of this section will be phrased in terms of curves in Euclidean space. Throughout this section, we assume all curves are piecewise smooth.
	
	We will describe two different ways to modify a curve in Euclidean space. The first involves changing the curve's value on a sub-interval.
	
	\begin{Def}
		Let $\gamma:[a, b]$ be any curve in $\R^k$. A curve $\eta:[a,b]\to \R^k$ is a \textit{path surgery of} $\gamma$ in \textit{surgery location} $(c,d)\subset [a, b]$ if $\gamma$ and $\eta$ agree outside of $(c,d)$.
	\end{Def}
	
	The second type of modification involves attaching a loop to a curve.
	
	\begin{Def}
		Let $\gamma:[a,b]\to \R^k$ be a curve. A curve $\eta:[a, b+c]\to\R^k$ is a \textit{length-}$c$ \textit{loop surgery of} $\gamma$ \textit{at} $t$ if $\gamma(s)=
		\begin{cases}
			\eta(s) & s\le t\\
			\eta(s+c) & s\ge t
		\end{cases}$. By a slight abuse of notation, we will also say $t$ is the \textit{location} of the loop surgery.
\end{Def}

Of course, general path and loop surgeries are of little interest. Any two curves with the same endpoints differ by one loop surgery followed by one path surgery, for instance. We will describe the types of path surgeries we wish to use in Lemma \ref{LongPerpendicularLengthGivesSurgeryFamily}. However, before we get there, some preliminary work is in order. We will need to perform one path or loop surgery for each half-space that we want the curve $q(\gamma)$ to visit, and we need to make sure that we can perform them all simultaneously, i.e., that after performing a surgery, we do not change the resulting curve again in the same location. We therefore need not just single surgeries, but collections of interchangeable surgeries, so that we can deal with surgery locations intersecting.

\begin{Def}

Let $\gamma:[a, b]$ be a curve in $\R^k$. A \textit{surgery family} $F$ is a set $m$ path surgeries $\eta_1, ... \eta_m$ of $\gamma$ located on a family of $m$ disjoint open sub-intervals of $[a,b]$. We call $\eta_1, ... \eta_m$ the \textit{options} of the family $F$. If $\eta_i$ is located on the open interval $I_i$, then $[a,b]\setminus \bigsqcup I_i$ consists of a component containing $a$, a component containing $b$, and $m-1$ either closed intervals or single points interior to $[a,b]$. The $m-1$ interior components will be called the \textit{interstices} of the family $F$.
\end{Def}

Again, there is no requirement that the $\eta_i$ have anything to do with one another as $i$ varies. In our case, the curves $\eta_i$ will all be chosen to visit a certain half-space and not lengthen $\gamma$ by too much, so that we think of them as interchangeable ways to modify our curve and attain an equally-good result.

Examining the interstices lets us bound how badly two path surgery families can overlap with one another, and use that bound to show that, given a collection of loop surgeries and path surgery families, we will be able to find a single curve incorporating each loop surgery and an option of each path surgery family as long as each family has sufficiently many options,.
 
\begin{prop}\label{SimultaneousSurgery}

Let $\gamma:[a,b]\to \R^k$ be a curve. Suppose we have $m_1$ loop surgeries $\eta_1, ... \eta_{m_1}$, such that $\eta_i$ is located at $t_i$ and has length $c_i$, and $m_2$ path surgery families $F_{1}, ... F_{m_2}$, such that $m_1+m_2\le n$ and each path surgery family has $n^2$ options. Denote the options of $F_i$ as $\eta_{i, j}$ located on intervals $I_{i,j}$. Then we can find a collection of options $\eta_{1, {j_1}}, ... \eta_{m_2, j_{m_2}}$ located on disjoint intervals, none of them containing the location of a path surgery. Moreover, there will be at least $2(n-i)$ options for the $i^{th}$ path surgery family we choose.

There will then be a single (piecewise smooth) curve $\eta$ in $\R^k$ consisting of $\gamma$ minus all the chosen path surgery locations, together with the images of $\eta_i|_{[t_i, t_i+c_i]}$ and $\eta_{i,j_i}|_{I_i, j_i}$.

\end{prop}

The choice of the letter $n$ in the statement of this proposition is to match the number of nilpotent factors in the group $G=\prod_{i=1}^n\mathbf{N}_i\rtimes\R^k$, since $n$ is the number of half-spaces there are to visit.

\begin{proof}
	
	We prove the first paragraph by induction on  $n$. Families $F_2$, ... $F_{m_2}$ each have $n^2-1=(n+1)(n-1)$ interstices. If the location $I$ of an option of $F_1$ meets at least $2n-1$ options of $F_j$ (for $j\ne i$), then $I$ must cover $2n-2$ interstices of $F_j$. There are therefore at most $\lfloor\frac{n^2-1}{2n-2}\rfloor=\lfloor \frac{n+1}{2}\rfloor$ options of $F_1$ that meet $2n-1$ or more options of each $F_j$. There are at most $n-1$ total loop surgeries and path surgery families other than $F_1$, and each loop surgery meets at most $1$ option for $F_1$. Therefore, there are at least $n^2-(n-1)\lfloor \frac{n+1}{2}\rfloor$ options of $F_1$ that miss every loop surgery location and touch no more than $2n-1$ options of any other family. We calculate that 
	\begin{align*}
		n^2-(n-1)\lfloor \frac{n+1}{2}\rfloor -(2n-2) & \ge n^2-\frac{n^2-1}{2}-(2n-2)\\
		&\ge \frac{n^2}{2}-2n+2\\
		&= \frac{1}{2}(n-2)^2 \ge 0.
	\end{align*}
	
	So there are indeed at least $2n-2$ options for the first choice, each of them meeting at most $2n-1$ options of the remaining families. Since $n^2-(2n-1)=(n-1)^2$ the result follows by induction. One simply throws out any excess options of the families $F_2$, ... $F_{m_2}$ until only $(n-1)^2$ options remain, and then repeats the argument.
	
	Once we have found this collection, there are no conflicts between path and loop surgeries, or between any pair of path surgeries. Therefore, we need to show that there are no conflicts between loop surgeries.
	
	Suppose that some of the loop surgeries have the same location, and permute them so that WLOG $t_1=t_2$. Since $\eta_1(t+c_1)=\gamma(s)$, and we can simply perform surgery $\eta_2$ at the point $t+c_1$ in the curve $\eta_1$ instead, obtaining a curve 
	$\eta(s)=\begin{cases} 
		\gamma(s) & s\le t \\
		\eta_1(s) & t\le s \le t+c_1\\
		\eta_2(s-c_1) & t+c_1\le s \le t+c_1+c_2\\
		\gamma(s-c_1-c_2) & s\ge t+c_1+c_2
	\end{cases}$. By the same token, we can apply any number of loop surgeries with the same location.
\end{proof}

When we perform a collection of path and loop surgeries as described above, we will say that the resulting curve $\eta$ \textit{agrees (on the complement of the path surgery locations) with} $\gamma$ \textit{up to a parameter shift}. We will often encapsulate the parenthesized text into notation instead of writing it out.

We now come to a description of the path surgeries we will perform on the curve $q(\gamma)$. The idea is that if a curve $\gamma(t)$ in Euclidean space traverses a very long distance perpendicular to the direction $\vec{v}$ while not changing its $\vec{v}$-coordinate much, then we can replace $\vec{v}(\gamma(t))$ with a different function $f(t)$ in such a way that the resulting curve is not much longer than $\gamma$, and that $f(t)$ attains a larger value than $\vec{v}(\gamma(t))$ did. 

We use the notation $\E^k$ here to emphasize that we now care about metric properties of the Euclidean space, not just the topological property of a curve being piecewise smooth. We will also fix the notation that if $V$ is a vector subspace of $\E^k$, then $\vec{v}|_V$ will denote the orthogonal projection of a vector $\vec{v}$ onto $V$.

\begin{lemma}\label{LongPerpendicularLengthGivesSurgeryFamily}
	
	Let $\gamma$ be any path in $\E^k$, and $\vec{v}$ any $k$-dimensional unit vector. Suppose that there is an interval $[a, b]$ along which $\int_a^b \lVert \vec{\gamma'(t)}|_{\vec{v}^\perp}\rVert\, dt \ge n^2 L$, and on which $\vec{v}(\gamma(t))\ge d$. Then there is a path surgery family with $n^2$ options $\eta_i$, such that each $\eta_i$ is no more than $1$ longer than $\gamma$, and such that each $\eta_i$ visits a point whose $\vec{v}$-coordinate is at least $d+\sqrt{\frac{L}{2}}$.
	
\end{lemma}

\begin{figure} [h]

	\centering

	\includegraphics{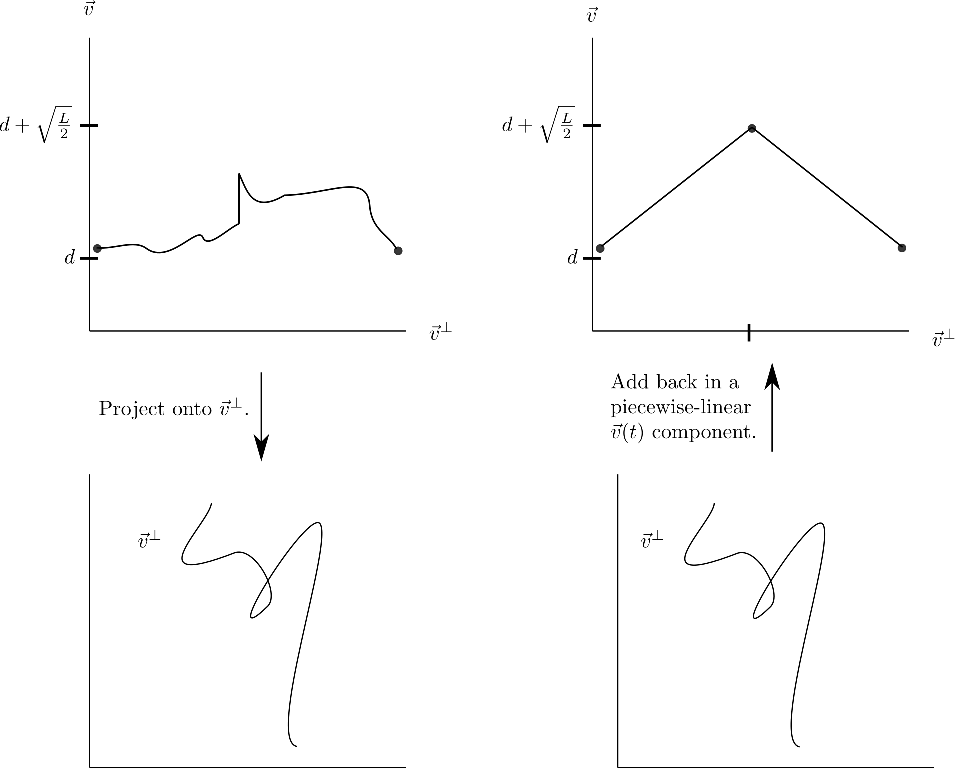}

	\caption{Schematic of the path surgery. On each of $n^2$ sub-intervals, we will preserve the $\vec{v}^{\perp}$-component of the path, and replace the $\vec{v}$ component of the path with a piecewise-linear function of the $\vec{v}^\perp$ arc length. If the curve is tangent to the $\vec{v}$-direction on some interval, projection down to $\vec{v}^\perp$ gets rid of this interval.}

\end{figure}
\begin{proof}
	
	For $i$ between $0$ and $n^2$, let times $t_i$ in $[a,b]$ be such that $\int _a^{t_i} \lVert \vec{\gamma'(t)}|_{\vec{v}^\perp}\rVert\, dt =iL$. If $\gamma'(t)$ is ever parallel to $\vec{v}$, the $t_i$ may not be uniquely determined, so choose the smallest such time for each. We will construct one surgery for each location $(t_i, t_{i+1})$. Notice that if the $\vec{v}$-coordinate of $\gamma(t)$ on any such interval is ever at least $d+\sqrt{\frac{L}{2}}$, then there is nothing to prove about this interval.
	
	So on the interval $[t_i, t_{i+1}]$, we may assume $\vec{v}(\gamma(t))$ is in $[d, d+\sqrt{\frac{L}{2}}]$. Denote $\gamma_{i, \vec{v}^\perp}:[t_i, t_{i+1}]\to\E^{k-1}$ to be $\gamma|_{\vec{v}^\perp}$ restricted to the interval $[t_i, t_{i+1}]$, and re-parameterize $\gamma_{i, \vec{v}^\perp}$ so that its arc length is linear on the interval. Let $f_i(t)$ be the piecewise function that interpolates linearly between $\vec{v}(\gamma(t_i))$ and $d+\sqrt{\frac{L}{2}}$ on the first half of the interval, and then interpolates linearly between $d+\sqrt{\frac{L}{2}}$ and $\vec{v}(\gamma(t_{i+1}))$ on the second half of the interval. That is, 
	\[f_i(t)=\begin{cases} \vec{v}(\gamma(t_i))\Big(1-\frac{t-t_i}{\frac{t_{i+1}-t_i}{2}}\Big) + \Big(d+\sqrt{\frac{L}{2}}\Big)\Big(\frac{t-t_i}{\frac{t_{i+1}-t_i}{2}}\Big) & t\le\frac{t_{i+1}+t_i}{2}\\
	\Big(d+\sqrt{\frac{L}{2}}\Big) \Big(2-\frac{t-t_i}{\frac{t_{i+1}-t_i}{2}}\Big) + \vec{v}(\gamma(t_{i+1})) \Big(\frac{t-t_i}{\frac{t_{i+1}-t_i}{2}}-1\Big) & t\ge \frac{t_{i+1}+t_i}{2}
	 \end{cases}.\]
	Take $\eta_i(t)$ to agree with $\gamma(t)$ outside of $(t_i, t_{i+1})$, and to be $\gamma_{i, \vec{v}^\perp}(t)+f_i(t)\vec{v}$ on $(t_i, t_{i+1})$, where we again identify the target of $\gamma_{i, \vec{v}^\perp}(t)$, a copy of $\E^{k-1}$, with the subspace $\vec{v}^\perp$. We must calculate the length of $\eta_i(t)$. Since $\eta_i(t)$ agrees with $\gamma$ outside of $(t_i, t_{i+1})$, we only need to compare the lengths on this interval.
	
	On each half of the interval $[t_i, t_{i+1}]$, $\eta_i(t)$ traverses distance in the $\vec{v}^\perp$ coordinates and in the $\vec{v}$ coordinate at linear speed. The total change in the $\vec{v}$ coordinate over each half is at most $\sqrt{\frac{L}{2}}$. Therefore,
	
	\begin{align*}
		\ell(\eta_i(t)|_{[t_i, \frac{t_i+t_{i+1}}{2}]}), \ell(\eta_i(t)|_{[\frac{t_i+t_{i+1}}{2}, t_{i+1}]}) & \le \sqrt{\Big(\frac{L}{2}\Big)^2+\sqrt{\frac{L}{2}}^2}\\
		&=\sqrt{\Big(\frac{L}{2}\Big)^2+\frac{L}{2}}\\
		&=\sqrt{\Big(\frac{L}{2}\Big)^2+2\Big(\frac{1}{2}\frac{L}{2}\Big)}\\
		&\le \sqrt{\Big(\frac{L}{2}\Big)^2+2\Big(\frac{1}{2}\frac{L}{2}\Big)+\frac{1}{4}}\\
		&= \sqrt{\Big(\frac{L}{2}+\frac{1}{2}\Big)^2}\\
		&=\frac{L}{2}+\frac{1}{2}\\
		&= \ell(\gamma_{i, \vec{v}^\perp}|_{[t_i, \frac{t_i+t_{i+1}}{2}]})+\frac{1}{2}, \ell(\gamma_{i, \vec{v}^\perp}|_{[\frac{t_i+t_{i+1}}{2}, t_{i+1}]})+\frac{1}{2}.
	\end{align*}
	
	Therefore, $\ell(\eta_i|_{[t_i, t_{i+1}]})\le \ell(\gamma_{i, \vec{v}^\perp})+1$. Since $\ell(\gamma|_{[t_i, t_{i+1}]})\le \ell(\gamma_{i, \vec{v}^\perp})$, it follows that $\ell(\eta_i)\le \ell(\gamma)+1$.	
\end{proof}

The following remark slightly restates Lemma \ref{LongPerpendicularLengthGivesSurgeryFamily} into the form we will use in the sequel.

\begin{rk} \label{PathSurgeryFamilyToReachSpecificHeight}
	
	Let $\gamma$ be a piecewise-smooth path in $\E^k$, $\vec{v}$ be a $k$-dimensional unit vector and $H$ a constant. If $\gamma(t)$ has $\vec{v}$-coordinate at least $H-r$ on an interval $[a,b]$ such that $\int_a^b \lVert\vec{\gamma'(t)}|_{\vec{v}^\perp}\rVert\, dt \ge 2n^2r^2$, then there is a family of surgeries of $\gamma$ with $n^2$ options $\eta_i$ so that each $\eta_i$ is no more than $1$ longer than $\gamma$ and such that each $\eta_i$ attains a point with $\vec{v}$ coordinate at least $H$. \qed
\end{rk} 

The number $H$ in question will be the boundary of a half space for the geodesic whose projection to $\E^k$ we are studying. A random piecewise-smooth path in $\E^k$ has no reason to contain long subcurves where a chosen coordinate does not vary much. Therefore, to obtain the setup needed to use Remark \ref{PathSurgeryFamilyToReachSpecificHeight}, we will need to use the fact that we started with a geodesic in the group $G$ and projected it to a Euclidean subspace. It is therefore time to re-introduce the group $G$ and describe surgery there.

\section{Existence of half-space visiting coarse geodesics.} \label{sec:DiagonalizableCaseCoarseGeodesics}

%In case I get confused about this point in the future:
%suppose $\int e^{-\alpha_i(\gamma(t))} \frac{d h_i(\gamma(t))}{dt} dt\ge k$. Then $\int e^{-\alpha_i(\gamma(t))} \frac{d h_i(\gamma(t))}{dt} dt\ge k$, and hence $\int \| \frac{d\gamma(t)}{dt}|_{\mathbf{N}}\|dt\ge k$.

%Now, by comparing to a straight line in Euclidean space,

%$\int \sqrt{  \| \frac{d\gamma(t)}{dt}|_{\mathbf{A}}\|^2 +\| \frac{d\gamma(t)}{dt}|_{\mathbf{N}}\|^2}dt \ge \sqrt{\int \| \frac{d\gamma(t)}{dt}|_{\mathbf{A}}\| dt ^2 +\int \| \frac{d\gamma(t)}{dt}|_{\mathbf{N}} \|dt ^2}$

Throughout this section, we retain the notation that if $V$ is a subspace of a vector space with an inner product, then $\vec{v}|_V$ denotes the orthogonal projection onto $V$. Here $V$ will be $\mathfrak{g}$, and the inner product will be a left-invariant Riemannian metric $g$ splitting $G$ perpendicularly.

In this section we will, by a slight abuse of notation, perform curve surgeries on curves $\gamma$ in $G$. When we do this, we mean to perform curve surgery on the underlying curve $q(\gamma)$. Moreover, we will always transform $\gamma$ into a half-space visiting curve, and take the box path $\eta$ in $G$ such that $q(\eta)$ is the desired surgery. In order to define half-space visiting uniformly, throughout this section, $g$ is a fixed metric on $G$, with $s_g:\R^k\to G$ its perpendicular splitting.

Of course, we can always perform a set of loop surgeries to render $q(\gamma)$ half-space visiting, but this may lengthen $q(\gamma)$ considerably. We will therefore use the following criterion to find a suitable way to replace a loop surgery with a path surgery family. 

\begin{lemma} \label{LongPerpendicularDistanceGivesSurgeryFamily}
Let $\gamma$ be a geodesic between $p$ and $q$ in $G$. Denote the heights of the half spaces $H_i$, and $|\nabla \alpha_i|=a_i$. Let $r>0$, and suppose that there is some sub-multicurve $\mathscr{C}$ of $q(\gamma)$, such that $\mathscr{C}$ lies above $\nabla\alpha_i$-height $H_i-r$, and $\int_\mathscr{C} \lVert\frac{d(q(\xi(t)))}{dt}|_{\nabla\alpha_i^\perp}\rVert \, dt\ge (n^2+1)8n^2r^2$.
%\begin{enumerate}
%	\item $R\ge \frac{(\sqrt{2}C+4r(n+1))(10n^2r^2)}{4r}$ 
%	\item $R \ge (8n^2C+n+1)10n^2r^2$.
%\end{enumerate}

Then there is a family of path surgeries with $n^2$ options so that each surgery lengthens $q(\gamma)$ by at most $1$, and each surgery reaches height $H_i$.
\end{lemma}

%\begin{rk} \label{OnlyOneIneq}
%Notice that as long as $r$ is not too small (e.g. $r\ge 1$ more than suffices), then the second inequality implies the first.
%\end{rk}

\begin{proof}
	
Suppose that, above $\nabla\alpha_i$-height $H_i-2r$, there is a connected subcurve $\mathscr{C}_1$ of $q(\gamma)$ such that \newline $\int_{\mathscr{C}_2} \lVert\frac{d(q(\xi(t)))}{dt}|_{\nabla\alpha_i^\perp}\rVert\,dt \ge8n^2r^2=2n^2(2r)^2$. Then we can apply Corollary \ref{PathSurgeryFamilyToReachSpecificHeight} to $\mathscr{C}_1$ to obtain the desired surgery family.

\begin{figure} [h]
	
	\centering
	\includegraphics{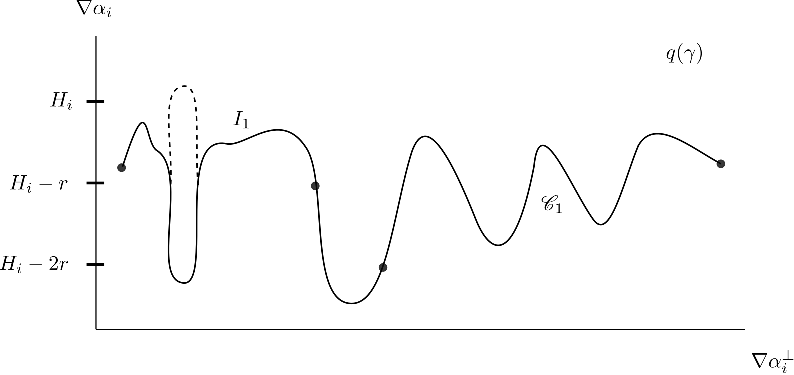}
	\caption{The cases of the proof are illustrated in this figure. The subcurves $I_1$ and $\mathscr{C}_1$ are both bounded by the indicated bold points. Either a curve like $\mathscr{C}_1$ must be long in coordinates perpendicular to $\nabla\alpha_i$, or else there must be at least $n^2$ many subcurves like $I_1$, for which reflecting the part below $H_i-r$ yields an equally-long curve that reaches height $H_i$.}
	\label{fig:LongPerpendicularDistanceGivesSurgeryFamily}
	
\end{figure}

So suppose that this does not happen. We find (connected) subcurves $I\in \mathscr{I}$ of $q(\gamma)$, such that $\int_{I\cap \mathscr{C}} \lVert \frac{dq(\gamma(t))}{dt}|_{\nabla\alpha_i^\perp}\rVert\, dt\ge 8n^2r^2$ for each one, and whose endpoints lie on $\mathscr{C}$. We can construct at least $\lfloor\frac{\int_\mathscr{C} \lVert\frac{d(q(\xi(t)))}{dt}|_{\nabla\alpha_i^\perp}\rVert\,dt}{8n^2r^2}\rfloor\ge n^2$ such curves.

Each of these at least $n^2$ curves has a point with $\nabla\alpha_i$-height of $H_i-r$ as well as points with $\nabla\alpha_i$-height no more than $H_i-2r$. Reflecting the part of the curve below $H_i-r$ about the line $\alpha_i=a_i(H_i-r)$, i.e. replacing the $\nabla\alpha_i(\gamma(t))$-coordinate on any such interval with $H_i-r+|H_i-r-\nabla\alpha_\gamma(t)|$, yields an isometric path which necessarily reaches height $H_i-r+|H_i-r-(H_i-2r)|=H_i$. So in this case, each surgery option does not lengthen $\gamma$ at all.

\end{proof}

It is not hard to find such sub-multicurves $\mathscr{C}$, as long as a geodesic $\gamma$ is far from half-space visiting in coordinate $i$. In the following lemma, $(\gamma_b^c)_i$ denotes $\gamma\cap \alpha_i^{-1}([a_ib, a_ic])$, where $a_i=|\nabla\alpha_i|$. That is $(\gamma_b^c)_i$ is the portion of $\gamma$ that lies between coordinates of $b$ and $c$ in the $\nabla\alpha_i$-direction.

We denote by $L_1$ a positive number sufficiently large so that all of the vector-space projections onto $\mathfrak{n}_i$ are $L_1$-Lipschitz. Note that the kernel of these projection map is spanned by all but one of the $\mathbf{n}_j$ (which depend only on $G$), as well as the Lie algebra of the copy of $\R^k$ perpendicular to $\mathfrak{n}$ (which depends on the metric). Therefore, both the projection maps themselves and their Lipschitz constants depend on the metric.

We will also need the following estimate.

\begin{rk}\label{ChangingSqrts}
	Let $a$, $m_1$, and $m_2$ all be greater than 0, such that $\sqrt{a^2+m_1^2}\le a+m_2$. Then $a\ge \frac{m_1^2-m_2^2}{2m_2}$. In particular, if $\sqrt{2}m_2\le m_1$, then $a\ge \frac{m_1^2}{4m_2}$.
\end{rk}

\begin{lemma} \label{ObtainingLongPerpendicularLength}

Let $\gamma$ be a geodesic between $p$ and $q$. Let $\xi$ be a half-space visiting $C$-rough box path that is a surgery of $\gamma$ at at most $n$ loci. Let $r>0$, $R\ge \sqrt{2}C$, and $b$ be any real number. Suppose that there is a sub-multicurve $\mathscr{C}$ of $(\gamma_b^{b+r})_i$, missing the path surgery locations, and so that $\int_{\mathscr{C}} \lVert\frac{dh_i(\gamma(t))}{dt}\rVert\, dt \ge L_1 R$. Suppose also that $\sqrt{2}C\le R$. Then $\int_{\mathscr{C}} \lVert \frac{d q(\gamma(t))}{dt}|_{\nabla\alpha_i^\perp}\rVert\,dt\ge \frac{R^2}{4C}-C-(n+1)r$. 
\end{lemma}

\begin{proof}    
Since $\xi$ agrees with $q(\gamma)$ up to a parameter shift on $\mathscr{C}$, removing all the movement in directions tangent to the $\mathbf{N}$-cosets from $\mathscr{C}$ cannot reduce the length of $\gamma$ by more than $C$, or else $\xi$ is shorter than $\gamma$. Therefore, we calculate

\begin{align*}
	\ell\circ q(\mathscr{C})+C & \ge \ell(\mathscr{C})\\
	&= \int_{\mathscr{C}} \sqrt{\lVert\frac{d\gamma(t)}{dt}|_{s_g(\R^k)}\rVert^2+\lVert\frac{d\gamma(t)}{dt}|_{\mathbf{N}}\rVert^2}\, dt\\
	&\ge \sqrt{\Big(\int_{\mathscr{C}} \lVert\frac{d\gamma(t)}{dt}|_{s_g(\R^k)}\rVert\, dt\Big) ^2+ \Big(\int_{\mathscr{C}} \lVert\frac{d\gamma(t)}{dt}|_{\mathbf{N}}\rVert\, dt\Big) ^2}\\
	&= \sqrt{\ell\circ q(\mathscr{C})^2+ \Big(\int_{\mathscr{C}} \lVert\frac{d\gamma(t)}{dt}|_{\mathbf{N}}\rVert\, dt \Big)^2}\\
	&\ge \sqrt{\ell\circ q(\mathscr{C})^2 + R^2}.
\end{align*}

By the assumption that $\sqrt{2}C\le R$ and Remark \ref{ChangingSqrts}, we see that $\ell\circ q(\mathscr{C})\ge \frac{R^2}{4C}$. 

\begin{figure} [t]
	
	\centering \includegraphics[scale=0.75]{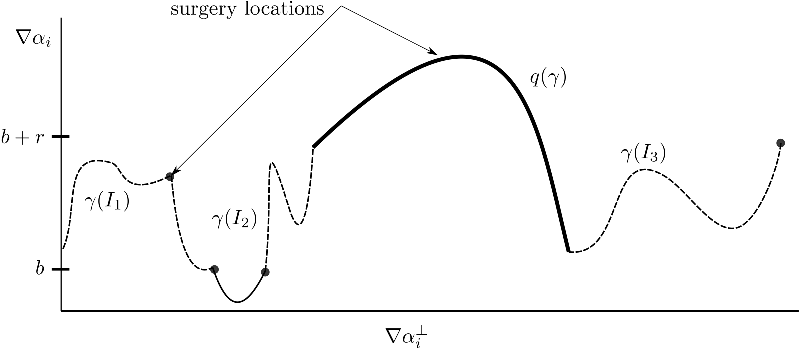}
	\includegraphics[scale=0.75]{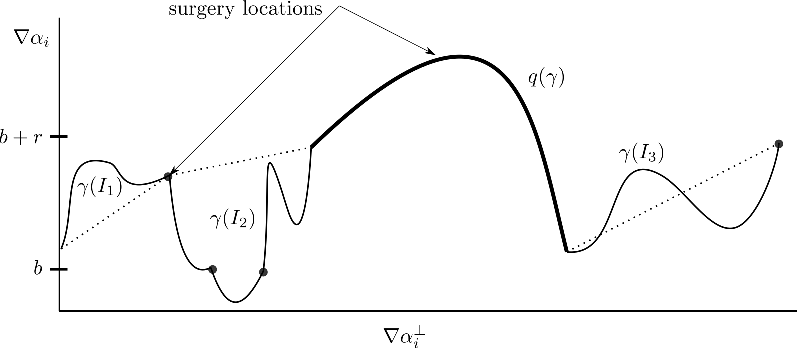}
	\caption{\textbf{Above:} The curve $\mathscr{C}$ is the dashed multicurve. The interval $I_1$ goes from the beginning of the curve until the location of the loop surgery. The interval $I_2$ goes from the loop surgery until the beginning of the path surgery, including parts of $q(\gamma)$ which are not in $\mathscr{C}$, and $I_3$ runs from the end of the path surgery until the end of the curve. \textbf{Below:} performing surgery on the bolded locations, and then changing $q(\gamma(t))$ on the intervals $I_1$, $I_2$, and $I_3$ to match the dotted lines, yields the curve $\xi_1$. Since $\ell\circ q(\mathscr{C})$ is long, the smaller $\int_\mathscr{C} \| \frac{d(q(\gamma(t))}{dt}|_{\nabla\alpha_i^\perp}\| dt$ is, the larger the variation of $\alpha_i$ must be on $\mathscr{C}$, despite the fact that $\mathscr{C}$ has $\nabla\alpha_i$-coordinate between $b$ and $b+r$. If this variation is too great, then $\xi_1$ is shorter than the geodesic $\gamma$.}
	
	\label{fig:LongPerpendicularDistance}
	
\end{figure}

An in figure \ref{fig:LongPerpendicularDistance}, let the intervals $I\in\mathscr{I}$ be all maximal intervals whose endpoints lie on the domain of $\mathscr{C}$ and in whose interiors do not meet any surgery locations of $\xi$. There are at most $n+1$ of these. Replace the $\nabla\alpha_i$-coordinate of $\xi$ on each $I\cap\mathscr{C}$ with a linear function (with respect to any parameterization; it does not matter which), and call the resulting curve $\xi_1$. 

We calculate 
\[\sum_{I\in\mathscr{I}} \int_{I\cap \mathscr{C}} \lVert\frac{d\xi_1(t)}{dt}|_{\nabla\alpha_i}\rVert\, dt\le (n+1)r.\]
But the intervals $I_k$ contain all of $\mathscr{C}$ jointly, so $\int_{ \mathscr{C}} \lVert\frac{d\xi_1(t)}{dt}
|_{\nabla\alpha_i}\rVert\, dt\le (n+1)r$, and so 
\[\ell(\xi_1|_{\mathscr{C}})\le \int_{\mathscr{C}} \lVert\frac{d\xi_1(t)}{dt}|_{\nabla\alpha_i^\perp}\rVert\,dt +(n+1)r.\]

If this is too much shorter than $\xi$ on $\mathscr{C}$, then $\xi_1$ is shorter than $\gamma$. That is
\[\int_{\mathscr{C}} \lVert\frac{d\xi_1(t)}{dt}|_{\nabla\alpha_i^\perp}\rVert\,dt +(n+1)r\ge \ell\circ q(\mathscr{C}) -C.\]
But $\xi_1(t)|_{\nabla\alpha_i^\perp}=\xi(t)|_{\nabla\alpha_i^\perp}$ and on $\mathscr{C}$, $q(\xi(t))=q(\gamma(t))$ (after parameter shift).
Therefore,
\[ \int_\mathscr{C} \lVert\frac{d\gamma(t)}{dt}|_{\nabla\alpha_i^\perp}\rVert\, dt\ge \frac{R^2}{4C}-C-(n+1)r\]
as required.
\end{proof}

Note that, as a consequence of the proof, we obtain a sub-multicurve $\mathscr{C}$ of $\gamma$ which is long perpendicular to $\nabla\alpha_i$ and also happens to miss the surgery locations of $\xi$. However, since the surgeries that Lemma \ref{LongPerpendicularDistanceGivesSurgeryFamily} gives us may overlap the surgery locations of $\xi$, this fact does not render the work in Proposition \ref{SimultaneousSurgery} superfluous. A modification of Lemma \ref{LongPerpendicularDistanceGivesSurgeryFamily} can be made to guarantee that the resulting surgery family again misses the surgery locations of $\xi$, but this still will not free us from needing to use Proposition \ref{SimultaneousSurgery}. Suppose the set of half-spaces that $\gamma$ is far from visiting is $H_1$, $H_2$, and $H_3$. As we will see in the proof Proposition \ref{DiagonalizableHSVProp}, we will perform a collection of surgeries on $\gamma$, and analyze the resulting coarse box geodesic $\xi_1$ using Lemmas \ref{ObtainingLongPerpendicularLength} and \ref{LongPerpendicularDistanceGivesSurgeryFamily}. This will give us a surgery family that lets us swap a long loop surgery for a short path surgery to reach half space $H_1$.  We will then use one member of this family to obtain a new coarse box geodesic $\xi_2$ and run the argument again for $H_2$.  However, past the second iteration, we will not only use one of the path surgeries to reach $H_2$ that we just found, but also we will be forced to pick a different surgery that makes $\gamma$ visit $H_1$. As a result, even if we modified Lemma \ref{LongPerpendicularDistanceGivesSurgeryFamily}, we could at best guarantee that the path surgery family we obtain misses the surgery locations of $\xi_i$. There would be no guarantee that it misses the choice of surgeries for any subsequent step, and thus we would need to use Proposition \ref{SimultaneousSurgery} anyway.

In order to apply Lemma \ref{ObtainingLongPerpendicularLength},  we need a sub-multicurve of a geodesic whose $\mathbf{N}_i$-length is large. However, it follows from Proposition \ref{AllNilpotentDirectionsAreEventuallyExponentiallyDistorted} that any path whose $\alpha_i$-value is much less than that of the half-space $H_i$ must have very long $\mathbf{N}_i$-length. Since $\alpha_i$ is a linear functional, this is equivalent to saying that if the path stays far enough from $H_i$, then it must have long $\mathbf{N}_i$-length. On the other hand, if the path does not stay far from $H_i$, then we can already modify it to visit $H_i$ by a loop surgery. We are therefore prepared to prove the main result. 

\begin{prop} \label{DiagonalizableHSVProp}
Let $G$ be a higher-rank Sol-type group with left-invariant Riemannian metric $g$ and $s_g$ a perpendicular splitting. Then there is a $K$ depending only on $g$ so that for any two points $p$ and $q$ in $G$, $p$ and $q$ are connected by a $K$-coarse HSV box geodesic. 
\end{prop}

%The only way that we will use the diagonalizability of the derivations is by way of Lemma \ref{PathsThatAvoidHoroballs}, and we will mark with a $(*)$ the inequalities that use it. These inequalities will involve considerable underestimates, to allow extra room when we generalize.

\begin{proof}
Let the associated half spaces for $p$ and $q$ be $\alpha_i^{-1}([H_i, \infty))$. Let $a$, $T$, and $C$ be the constants in Corollary \ref{UniformEventualExponentialDistortion}. Take $1>\epsilon>0$ so that $(a(1-\epsilon))>1$, and let $N=\sum_{j=0}^\infty (1-\epsilon)^{j}$. Let $L$ be the constant so that each vector space projection onto the $\mathfrak{n}_i$ is $L$-Lipschitz as before. 

%Denote by $L_2$ the constant in Lemma \ref{PathsThatAvoidHoroballs}. Let $\lVert\nabla\alpha_i\rVert=a_i$.

Recall that for $m$ a positive integer, $m!!$ is defined as the product of positive integers no more than $m$ and equal to $m$ in parity, i.e., $m!!= m(m-2)(m-4)...$. Choose $r\ge\max\{1,T\}$ sufficiently large so that for $r'\ge r$ and $C'=n(2r'+1)$, the following 3 conditions hold:

\begin{enumerate}
	\item $\frac{Ca^{r'}}{N} \frac{(2n-3)!!}{(2n-2)!!}\ge \sqrt{2}C'$,
	\item $C^2a^{2r'} \ge \frac{(2n-2)!!^2}{(2n-3)!!^2}4C'N^2L^2 \big(32n^2r'^2(n^2+1)+C'+(n+1)r'\big)$,
	\item for $j\ge 1$, $(a(1-\epsilon))^{2(j-1)r'}\ge j^2$.
\end{enumerate}

We will show the proposition for $K=(n+1)(2r+1)$.

Let $\gamma$ denote the geodesic between $p$ and $q$, and denote $r_i$ the distance from $\gamma$ to $\alpha_i^{-1}([H_i, \infty))$. Re-order the coordinates if necessary so that the $r_i$ are in decreasing order, and suppose that $r_1\ge...r_m\ge r > r_{m+1}...$ Also, let the notation $(\gamma_b^c)_i$ denote $\gamma\cap \alpha_i^{-1}([a_ib, a_ic])$ as before.

Let $x, y$ be elements of $\mathbf{N}_i$. Denote by $d_{i,s}(x,y)$ the minimal distance of a path between $(x, \vec{v})$ and $(y, \vec{v})$ tangent to $\mathbf{N}_i\vec{v}$ where $\vec{v}$ is some (hence any) vector with $\alpha_i(\vec{v})=s$. By definition, $d_{i, H_i}(p_i, q_i)=1$, where $p_i$ and $q_i$ are the $\mathbf{N}_i$-coordinated of $p$ and $q$. The content of Lemma \ref{PathsThatAvoidHoroballs} is that $d_{i, H_i}(p_i, q_i)$ grows at least exponentially as we decrease $H_i$.

Fix the following notation: If $\mathscr{C}$ is a multicurve, denote by $S$ the set of maximal subcurves of $\mathscr{C}$, each parameterized to take the interval into G. Denote by $\eta_i$, for $\eta\in S$, the $\mathbf{N}_i$-coordinate function of $\eta$. Then $\mathscr{L}_{i, s}(\mathscr{C})$ is defined to be
\[\sum_{\eta\in S} d_{i, s}(\eta_i(0), \eta_i(1)).\]

For each $1\le i\le m$, decompose $\gamma$ into the multicurves $(\gamma^{H_i-jr_i}_{H_i-(j+1)r_i})_i$, where $j$ ranges from $1$ to $\infty$. See figure \ref{fig:SlicingCurve1}. For each $i$ there must be some $j_i$ so that 
\[\mathscr{L}_{i, H_i}((\gamma^{H_i-jr_i}_{H_i-(j+1)r_i})_i) \ge \frac{(1-\epsilon)^{(j_i-1)r_i}}{N},\]
because $r_i>1$, which implies that $N\ge\sum_{j=0}^\infty(1-\epsilon)^{jr_i}$. 

\begin{figure} [H]
	
	\centering
	\includegraphics{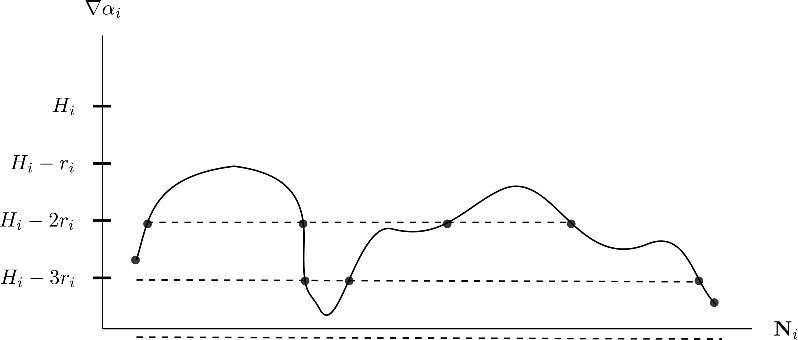}
	\caption{Slicing the geodesic $\gamma$ into pieces. $d_{i, H_i}$-distance is shown on the $x$-axis. One of the slices must traverse a sufficiently large fraction of the $d_{i, H_i}$ distance far enough below the half-space $H_i$ that we can use Lemma \ref{ObtainingLongPerpendicularLength}.}
	\label{fig:SlicingCurve1}
	
\end{figure}

We will now use these sub-multicurves one at a time, starting with $(\gamma^{H_1-j_1r_1}_{H_1-(j_1+1)r_1})_1$. Since $(\gamma^{H_1-j_1r_1}_{H_1-(j_1+1)r_1})_1$ lies below $\alpha_1^{-1}(H_1-j_1r_1)$, it follows from Corollary \ref{UniformEventualExponentialDistortion} that

\begin{equation*}
	\int_{(\gamma^{H_1-j_1r_1}_{H_1-(j_1+1)r_1})_1} \lVert\frac{dh_1(\gamma(t))}{dt}\rVert\, dt\ge \frac{Ca^{j_1r_1}(1-\epsilon)^{(j_1-1)r_1}}{N},
\end{equation*}

where the norm $\lVert\cdot \rVert$ denotes the $g$-length of the vector $\frac{dh_i(\gamma(t))}{dt}$ at the point $\gamma(t)$. Since $r_1>r\ge 1$, we conclude that 
\[\int_{(\gamma^{H_1-j_1r_1}_{H_1-(j_1+1)r_1})_1}  \lVert\frac{dh_1(\gamma(t))}{dt}\rVert\, dt\ge \frac{Ca^{r_1}(a(1-\epsilon))^{(j_1-1)r_1}}{N},\]
which is all we will use going forward.

Take $\xi_1$ to be the HSV box path that arises from $n$ loop surgeries on $\gamma$. At the maximum of each $\alpha_i$, we attach a (possibly degenerate) loop that travels in the direction $\nabla\alpha_i$ a distance $r_i$ to reach the half space, then performs all the motion in the $h_i$ direction, and then returns. Note that since $r_1\ge r_i$ for all other $i$, $\xi_1$ is $n(2r_1+1)$-roughly geodesic. Since $r_1>r$, the assumptions now apply to $r_1=r'$ and $C_1=n(2r_1+1)$.

We wish to apply Lemma \ref{ObtainingLongPerpendicularLength} to the curve $(\gamma^{H_1-j_1r_1}_{H_1-(j_1+1)r_1})_1$. To do this, note that since $a(1-\epsilon)>1$, 
\[\int_{(\gamma^{H_1-j_1r_1}_{H_1-(j_1+1)r_1})_1} \lVert\frac{dh_1(\gamma(t))}{dt}\rVert\, dt\ge \frac{Ca^{r_1}}{N},\]
which is at least $\sqrt{2}C_1$ by assumption (1).

Therefore, applying Lemma \ref{ObtainingLongPerpendicularLength} yields that 
\[\int_{(\gamma^{H_1-j_1r_1}_{H_1-(j_1+1)r_1})_1} \lVert \frac{d q(\gamma(t))}{dt}|_{{\nabla\alpha_1}^\perp}\rVert\,dt \ge
\frac{C^2a^{2r_1}(a(1-\epsilon))^{2(j_1-1)r_1}}{4C_1N^2L^2}-C_1-(n+1)r_1.\]
%\frac{L_2^2e^{ar_1}(\sqrt{e}(1-\epsilon))^{2a(j_1-1)r_1}}{4C_1N^2L_1}-C_1-(n+1)r_1.\]

To obtain a surgery family for coordinate $\nabla\alpha_1$, we want to apply Lemma \ref{LongPerpendicularDistanceGivesSurgeryFamily}. 

By assumption (2), we calculate that

\begin{align*}
	C^2a^{2r_1} &\ge \frac{(2n-2)!!^2}{(2n-3)!!^2}4C_1N^2L^2 (32n^2r_1^2(n^2+1)+C_1+(n+1)r_1)\\
	&\ge 4C_1N^2L^2(32n^2r_1^2(n^2+1)+C_1+(n+1)r_1)\\
	\frac{C^2a^{2r_1}}{4C_1N^2L^2} & \ge 32n^2r_1^2(n^2+1)+C_1+(n+1)r_1\\
	&=(n^2+1)8n^2(2r_1)^2+C_1+(n+1)r_1.
\end{align*}

Applying assumption (3), \[\frac{C^2a^{2r_1}(a(1-\epsilon))^{2(j_1-1)r_1}}{4C_1N^2L^2} \ge (n^2+1)8n^2((2j_1)r_1)^2+j_1^2(C_1+(n+1)r_1).\]
Since $j_1^2>1$ and all other constants are positive, we can simplify to
\[ \frac{C^2a^{2r_1}(a(1-\epsilon))^{2(j_1-1)r_1}}{4C_1N^2L^2} \ge (n^2+1)8n^2((2j_1)r_1)^2+C_1+(n+1)r_1.\]
And since $2j_1\ge j_1+1$, it follows that
\[\frac{C^2a^{2r_1}(a(1-\epsilon))^{2(j_1-1)r_1}}{4C_1N^2L^2} \ge (n^2+1)8n^2((j_1+1)r_1)^2+C_1+(n+1)r_1.\]
We rearrange to determine that 
\[\int_{(\gamma^{H_1-j_1r_1}_{H_1-(j_1+1)r_1})_1} \lVert \frac{d q(\gamma(t))}{dt}|_{{\nabla\alpha_1}^\perp}\rVert\,dt \ge (n^2+1)8n^2((j_1+1)r)^2.\]

Therefore, we can apply Lemma \ref{LongPerpendicularDistanceGivesSurgeryFamily} to find a surgery family with $n^2$ options, each of which lengthens $q(\gamma)$ by no more than a length of $1$ and each of which attains $\nabla\alpha_1$-height of $H_1$.

\begin{figure} [h]
	
	\centering
	\includegraphics{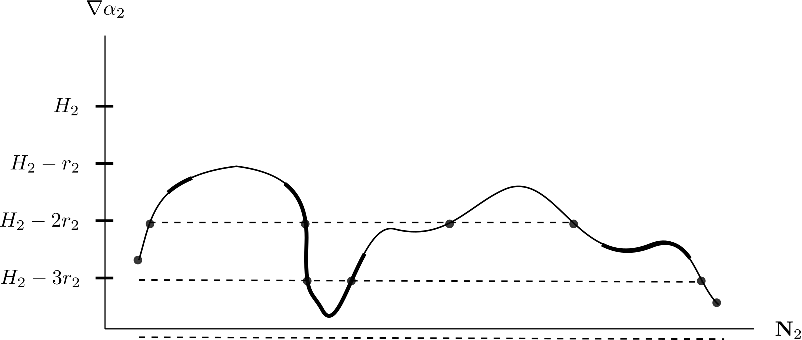}
	\caption{Determining surgeries for the curve $\xi_1$. The cross-section is the same as before, but bolded are surgery locations of options of the first surgery family. As long as there are enough options for each family, we can choose options that do not cover much of the desired slice of the curve. For instance, supposing that $j_2=1$, we could choose the rightmost surgery option.}
	\label{fig:SlicingCurve2}
	
\end{figure}

We will repeat this process $m-1$ more times by finding a sequence of curves $\xi_i$. Suppose we have already found $i$ path surgery families with $n^2$ options, lengthening $q(\gamma)$ by at most $1$ to visit the first $i$ half-spaces. We also have $n-i$ loop surgeries as before to make $q(\gamma)$ visit the other $n-i$ half-spaces. Using Proposition \ref{SimultaneousSurgery}, we choose an option from each of these $i$ families that do not conflict. We are guaranteed $2n-2$ options for the first interval $I_1$, $2n-4$ options for the second interval $I_2$, and so on. We therefore choose $I_1$ as shown in figure \ref{fig:SlicingCurve2}, so that 
\[\mathscr{L}_{i+1, H_{i+1}} ((\gamma_{H_{i+1}-(j_{i+1}+1)r_i}^{H_{i+1}-(j_{i+1})r_i})_{i+1}\cap I_1)\le \frac{\mathscr{L}_{i+1, H_{i+1}} ((\gamma_{H_{i+1}-(j_{i+1}+1)r_i}^{H_{i+1}-(j_{i+1})r_i})_{i+1}}{2n-2},\]
so that at least $ \frac{2n-3}{2n-2}\mathscr{L}_{i+1, H_{i+1}} ((\gamma_{H_{i+1}-(j_{i+1}+1)r_i}^{H_{i+1}-(j_{i+1})r_i})_{i+1}$ remains in the complement $\gamma_{H_{i+1}-(j_{i+1}+1)r_i}^{H_{i+1}-(j_{i+1})r_i})_{i+1}\setminus I_1$. Then choose the second option so that
\[\mathscr{L}_{i+1, H_{i+1}} ((\gamma_{H_{i+1}-(j_{i+1}+1)r_i}^{H_{i+1}-(j_{i+1})r_i})_{i+1}\cap I_2)\le \frac{1}{2n-4} \mathscr{L}_{i+1, H_{i+1}} ((\gamma_{H_{i+1}-(j_{i+1}+1)r_i}^{H_{i+1}-(j_{i+1})r_i})_{i+1}\setminus I_1),\]
so that 
\[ \mathscr{L}_{i+1, H_{i+1}} ((\gamma_{H_{i+1}-(j_{i+1}+1)r_i}^{H_{i+1}-(j_{i+1})r_i})_{i+1}\setminus(I_1\cup I_2))\ge\frac{(2n-3)(2n-5)}{(2n-2)(2n-4)}\mathscr{L}_{i+1, H_{i+1}} ((\gamma_{H_{i+1}-(j_{i+1}+1)r_i}^{H_{i+1}-(j_{i+1})r_i})_{i+1}\]
and so on. After $i$ steps, at least $\frac{(2n-3)!!}{(2n-2)!!}$ of the length must remain.

We thus obtain a curve $\xi_{i+1}$ which is $C_{i+1}=(n)(2r_i+1)$-coarsely geodesic, agreeing (up to a parameter shift) with $\gamma$ on a sub multicurve $\mathscr{C}_{i+1}=(\gamma_{H_{i+1}-(j_{i+1}+1)r_i}^{H_{i+1}-(j_{i+1})r_i})_{i+1}\setminus\bigsqcup I_{l} $ such that 
\[\mathscr{L}_{i+1, H_{i+1}} (\mathscr{C}_{i+1})\ge \frac{(1-\epsilon)^{(j_{i+1}-1)r_{i+1}}}{N}\frac{(2n-3)!!}{(2n-2)!!}.\]

Again by Corollary \ref{UniformEventualExponentialDistortion},

\[ \int_{\mathscr{C}_{i+1}} \lVert\frac{dh_{i+1}(\gamma(t))}{dt} \rVert\, dt\ge\frac{Ca^{r_{i+1}}(a(1-\epsilon))^{(j_{i+1}-1)r_{i+1}}}{N}\frac{(2n-3)!!}{(2n-2)!!}. \]

By assumption (1) this is at least $\sqrt{2}C_{i+1}$ so that we can apply Lemma \ref{ObtainingLongPerpendicularLength} to show that 
\[\int_{\mathscr{C}_{i+1}} \lVert\frac{d q(\gamma(t))}{dt}|_{\nabla\alpha_{i+1}^\perp}\rVert\,dt \ge \frac{C^2a^{2r_{i+1}}(a(1-\epsilon))^{2(j_{i+1}-1)r_{i+1}}}{4C_{i+1}N^2L^2}\frac{(2n-3)!!^2}{(2n-2)!!^2}-C_{i+1}-(n+1)r_{i+1}.\]
By assumption (2), 
\[\frac{C^2a^{2r_{i+1}}}{4C_{i+1}N^2L^2} \ge \frac{(2n-2)!!^2}{(2n-3)!!^2}(32n^2r_{i+1}^2(n^2+1)+C_{i+1}+(n+1)r_{i+1}),\]
and the argument concludes in the same way as the $i=1$ case that we can again apply Lemma \ref{LongPerpendicularDistanceGivesSurgeryFamily}.
\end{proof}

\section{Intrinsic rough isometries} \label{sec:IntrinsicRIs}

In this Section, we will prove Theorem \ref{IntrinsicRIsThmIntro}.

\begin{Def}

Let $G$ be a Lie group, and $f:G\to G$ be a quasi-isometry for some (and therefore all) left-invariant Riemannian metrics on $G$. Then $f$ is an \textit{intrinsic rough isometry} if $f$ is a rough isometry with respect to every left-invariant Riemannian metric. Note that the implied constant depends on the choice of metric.
\end{Def}

Because our theorem has the requirement that the metrics considered must split $G$ perpendicularly, we are not able to conclude that quasi-isometries are intrinsic rough isometires. We must also work modulo a conjecture about the structure of quasi-isometries of our groups.

\begin{conj}\label{GeneralizedPengQIs}
Let $G\cong \prod_{i=1}^n \mathbf{N}_i \rtimes \R^k$ be nondegenerate with $n\ge 2$. Let $q:G\to G$ be a quasi-isometry for some (hence any) left-invariant Riemannian metric, and let $\overline{\R^k}$ be any lift of the quotient map $G\to \R^k$. Give the $G$ semi-direct normal coordinates with the factor $\overline{\R^k}$ on the right. Then $q$ is at finite distance from a map $L_x\circ (\prod_{i=1}^n f_i)\circ \sigma$, where $L_x$ is a left-translation by a group element, each $f_i$ is a map from $\mathbf{N}_i$ to itself, and $\sigma$ is an element of a finite group of symmetries which simultaneously permute the $\mathbf{N}_i$, act on $\overline{\R^k}$ by a linear isomorphism which permutes the associated roots, such that if $\sigma^*\alpha_i=\alpha_j$, then $\sigma(\mathbf{N}_i)=\mathbf{N}_j$ and $(\mathbf{N}_i, D_i)$ is isomorphic to $(\mathbf{N}_j, D_j)$. 
\end{conj}

This conjecture is expected to hold. Peng has proved the result under the additional assumptions that the $\mathbf{N}_i$ are abelian, and that the group $G$ is unimodular \cite{Peng1, Peng2}. In rank-1, the result is known due to Eskin--Fisher--Whyte without the unimodularity assumption, and in fact this case is easier than the unimodular case \cite{EFW1, EFW2}. So the conjecture is that one can combine Eskin--Fisher--Whyte's treatment of the non-unimodular case with Peng's work to remove the unimodular assumption, and that the work can be extended to nilpotent groups as well. See also Ferragut's proof of a similar result in the setting of horocyclic products \cite{Ferragut}.

The element $\sigma$ is denoted to lie in $Sym(G)$, the group of all possible permutations of roots and root spaces as above. The group $Sym(G)$ is typically trivial. For instance, any time the $(\mathbf{N}_i, D_i)$ are in pairwise distinct isomorphism classes, there is no hope for such an isomorphism.

We cannot show that every quasi-isometry is an intrinsic rough isometry at present. However, we can show that those with $\sigma$ trivial (whether or not $Sym(G)$ is trivial on the whole) are rough isometries for metrics that split $G$ perpendicularly.

\begin{thm} \label{IntrinsicRIs}

Assume Conjecture \ref{GeneralizedPengQIs} holds. Let $G\cong \prod_{i=1}^n \mathbf{N}_i \rtimes \R^k$ be nondegenerate, $q:G\to G$ be a quasi-isometry of $(G, d_g)$, where $g$ is a left-invariant Riemannian metric that splits $G$ perpendicularly. Decompose $q$, up to finite distance, into $L_x\circ (\prod_{i=1}^n f_i)\circ \sigma$, with respect to the lift $s_g(\R^k)$, and suppose $\sigma$ is trivial. Then $q$ is a rough isometry for $(G, d_g)$.
\end{thm}

Of course, in the case where $G$ is abelian-by-abelian and unimodular, the Conjecture need not be assumed because it is already a theorem. So this result subsumes Theorem \ref{IntrinsicRIsThmIntro}.

\begin{proof}
By Proposition \ref{DiagonalizableHSVProp}, there is a $C_1$ depending on $g$ so that any two points $x$ and $y$ can be connected by a $C_1$-rough box geodesic. Let $C_2$ be the distance between $q$ and the map $L_x\circ (\prod_{i=1}^n f_i)$. Denote $q'=(\prod_{i=1}^n f_i)$, and suppose $q'$ is a $(K, C_3)$ quasi-isometry. 

$q'$ sends box paths to box paths, and does not change the length of any segment in a coset of $s_g(\R^k)$. Between any two points $x$ and $y$ there is a $C_1$-rough box geodesic $\gamma$ so that $\gamma$ is composed of at most $n+1$ (maximal) segments $\gamma_{2i}$ tangent to cosets of $s_g(\R^k)$, together with at most $n$ additional subsegments $\gamma_{2i+1}$, each of length at most $1$, where $i$ ranges from $0$ to $n$. After applying $q'$, the lengths of the $\gamma_{2i}$ do not change, and the endpoints of the $\gamma_{2i+1}$ may be at most $K+C_3$ apart. So the path between $q'(x)$ and $q'(y)$ that agrees with $q'(\gamma_{2i})$ and replaces $q'(\gamma_{2i+1})$ with a geodesic between its endpoints has length at most $\ell(\gamma)+n(K+C_3)$. Therefore,
\[d_g(q'(x), q'(y))\le d(x,y)+C_1+n(K+C_3).\]
Since $q$ is distance $C_2$ from a composition of $q'$ with an isometry, it follows that 
\[ d_g(q(x), q(y))\le d(x, y)+C_1+2C_2+n(K+C_3).\]
\end{proof}

On the other hand, some quasi-isometries with $\sigma$ nontrivial are rough isometries and some are not. For example, consider $G=\R^4\rtimes\R^2$, where we denote a basis for $\R^4$ $\{x_i: i\in [1, 4]\}$ and a $s_g^* g$-orthonormal basis for $\R^2$ $\{e_1, e_2\}$. Suppose the roots $\alpha_1$ and $\alpha_3$ are $\pm e_1$, and the the roots $\alpha_2$ and $\alpha_4$ are $\pm 2e_2$. Choose a Riemannian metric $g$ so that this basis is orthonormal. We consider two maps whose representations as product maps of the form described in Conjecture \ref{GeneralizedPengQIs} is just an element of the symmetric group (i.e. $L_{Id}\circ \prod_{i=1}^n Id\circ \sigma$ where the map $\sigma$ is not the identity). The map $\sigma_1:G\to G$ defined on the $\R^4$-coordinates by $\begin{bmatrix} 0 & 1 & 0 & 0 \\ 1 & 0 & 0 & 0 \\ 0 & 0 & 0 & 1\\ 0 & 0 & 1 & 0    
\end{bmatrix}$, and on the $\R^2$-coordinates as $\begin{bmatrix}
0 & \frac{1}{2}\\ 2 & 0
\end{bmatrix}$ is a quasi-isometry, but it does not induce the identity on $g|_{\R^2}$. The proof of Theorem \ref{IntrinsicRIs} does not go through for $\sigma_1$, and indeed $\sigma_1$ is not a rough isometry for the metric $d_g$ because it fails to induce a rough isometry on the totally-geodesic copy of $\R^2$. On the other hand, the map $\sigma_2:G\to G$ defined on $\R^4$ coordinates by $\begin{bmatrix} 0 & 0 & 1 & 0 \\ 0 & 0 & 0 & 1 \\ 1 & 0 & 0 & 0 \\ 0 & 1 & 0 & 0\end{bmatrix}$, and on $\R^{2}$-coordinates as $\begin{bmatrix} - 1 & 0 \\ 0 & -1\end{bmatrix}$. Then $\sigma_2$ is an isometry, and \textit{a fortiori} a rough isometry, for $d_g$.

\bibliographystyle{plain}
\bibliography{Left-invariant_Riemannian_distances_on_higher-rank_Sol-type_groups.bib}

\end{document}